\documentclass[reqno]{amsart}

\usepackage{amsmath}
\usepackage{amsfonts}
\usepackage{amssymb,enumerate,comment,mathdots}
\usepackage{amsthm}
\usepackage[all]{xy}
\usepackage{rotating}
\usepackage{hyperref}
\usepackage{color}

\usepackage{graphicx}

\theoremstyle{plain}
\newtheorem{lem}{Lemma}[section]
\newtheorem{cor}[lem]{Corollary}
\newtheorem{prop}[lem]{Proposition}
\newtheorem{thm}[lem]{Theorem}

\theoremstyle{definition}
\newtheorem{defn}[lem]{Definition}
\newtheorem{ex}[lem]{Example}

\newtheorem{rmk}[lem]{Remark}

\newtheorem{notn}[lem]{Notation}
\newtheorem{fact}[lem]{Fact}

\theoremstyle{remark}
\newtheorem{case}{Case}
\newtheorem{subcase}{Case}[case]

\newcommand{\mingen}{\operatorname{mingen}}

\newcommand{\Hom}{\operatorname{Hom}}

\newcommand{\s}{\mathfrak{S}}
\newcommand{\tor}{\operatorname{Tor}}
\newcommand{\im}{\operatorname{Im}}

\newcommand{\Cl}{\operatorname{Cl}}

\newcommand{\Ker}{\operatorname{Ker}}

\newcommand{\ideal}[1]{\mathfrak{#1}}
\newcommand{\m}{\ideal{m}}

\newcommand{\p}{\ideal{p}}
\newcommand{\q}{\ideal{q}}

\newcommand{\fa}{\ideal{a}}
\newcommand{\fb}{\ideal{b}}

\newcommand{\sfk}{\mathsf k}

\newcommand{\ti}{\tilde}

\newcommand{\bbz}{\mathbb{Z}}
\newcommand{\bbn}{\mathbb{N}}

\newcommand{\from}{\leftarrow}
\newcommand{\xra}{\xrightarrow}

\newcommand{\vf}{\varphi}

\newcommand{\x}{\mathbf{x}}

\newcommand{\gb}{\beta}

\newcommand{\gd}{\delta}

\newcommand{\gl}{\lambda}

\renewcommand{\geq}{\geqslant}
\renewcommand{\leq}{\leqslant}
\renewcommand{\nleq}{\nleqslant}
\renewcommand{\ker}{\Ker}

\newcommand{\Ext}[4][R]{\operatorname{Ext}_{#1}^{#2}(#3,#4)}

\newcommand{\Otimes}[3][R]{#2\otimes_{#1}#3}
\renewcommand{\Hom}[3][R]{\operatorname{Hom}_{#1}(#2,#3)}	
\newcommand{\Tor}[4][R]{\operatorname{Tor}^{#1}_{#2}(#3,#4)}

\newcommand{\ssm}{\smallsetminus}

\numberwithin{equation}{lem}

\begin{document}

\bibliographystyle{amsplain}

\author{Sean K. Sather-Wagstaff}

\address{Department of Mathematical Sciences,
Clemson University,
O-110 Martin Hall, Box 340975, Clemson, SC 29634
USA}

\email{ssather@clemson.edu}

\urladdr{https://ssather.people.clemson.edu/}

\author{Tony Se}

\address{Department of Mathematics,
University of Mississippi,
Hume Hall 305, P.O. Box 1848, University, MS 38677
USA}

\email{ttse@olemiss.edu}

\urladdr{http://math.olemiss.edu/tony-se/}

\author{Sandra Spiroff}

\address{Department of Mathematics,
University of Mississippi,
Hume Hall 305, P.O. Box 1848, University, MS 38677
USA}

\email{spiroff@olemiss.edu}

\urladdr{http://math.olemiss.edu/sandra-spiroff/}

\thanks{
Sandra Spiroff was supported in part by Simons Foundation Collaboration Grant 245926.}

\title{On semidualizing modules of ladder determinantal rings} 

\date{\today}


\keywords{divisor class group, ladder determinantal ring, semidualizing module}
\subjclass[2010]{
13C20, 
13C40, 
}

\begin{abstract} 
We identify all semidualizing modules over certain classes of ladder determinantal rings over a field $\sfk$. Specifically, given a ladder of variables $Y$, we show that the ring $\sfk[Y]/I_t(Y)$ has only trivial semidualizing modules up to isomorphism 
in the following cases: (1) $Y$ is a one-sided ladder, and (2) $Y$ is a two-sided ladder with $t=2$ and no coincidental inside corners. 
\end{abstract}

\maketitle



\section*{Introduction} \label{sec0}


Let $R$ be a commutative noetherian ring and let $\sfk$ be a field.  A finitely generated $R$-module $C$ is \textbf{semidualizing} if $\Hom CC\cong R$ and $\Ext iCC=0$ for all $i\geq 1$. 
The set of isomorphism classes of semidualizing $R$-modules is denoted $\s_0(R)$.
See Section~\ref{sec180114a} for background information on these modules.

Semidualizing modules arise in several different contexts.
Hans-Bjorn Foxby~\cite{foxby:gmarm} introduced them to provide a useful generalization of the dualities with respect to a free module of rank 1 
and with respect to a dualizing/canonical module. 
Other applications include progress 
by Luchezar Avramov and Foxby~\cite{avramov:rhafgd} and Sean Sather-Wagstaff~\cite{sather:cidfc}
on composition questions for local ring homomorphisms, and advances on a question of Craig Huneke on growth of Bass numbers of
local rings by Sather-Wagstaff~\cite{sather:bnsc}. 

Despite the utility of semidualizing modules, very little is known about the set $\s_0(R)$.
Only recently have Saeed Nasseh and Sather-Wagstaff~\cite{nasseh:gart} shown that this set is finite.
Anders Frankild and Sather-Wagstaff show that the set has even cardinality when $R$ is local, complete, Cohen-Macaulay, and not Gorenstein
in~\cite{frankild:sdcms}.
At this time, we only have more information than this in very special cases: 
Olgur Celikbas and Hailong Dao~\cite{celikbas:ncdfcmlr} deal with certain Veronese subrings;
William Sanders~\cite{MR3433990} handles some rings of invariants;
Sather-Wagstaff treats determinantal rings in~\cite{SSW};
and Nasseh, Sather-Wagstaff, and Ryo Takahashi~\cite{nasseh:vetfp,nasseh:hplrq} handle the rings that specialize to non-trivial fiber products
(this includes the well-known but seemingly undocumented result for rings of minimal multiplicity).

In particular, the following question~\cite[Question~4.13]{SSW} of Sather-Wagstaff is still open:
If $R$ is a local ring, must the cardinality $|\s_0(R)|$ be a power of 2?
Each of the special cases in the previous paragraph answers this question in the affirmative for its certain class of rings.
In fact, in most cases the rings admit only \textbf{trivial semidualizing modules}, namely, the free module of rank 1 and a dualizing module;
exceptions occur for determinantal rings with coefficients in non-Gorenstein rings.

We provide 
more special-case
evidence of an affirmative answer to Sather-Wag\-staff's question
by studying the semidualizing modules of ladder determinantal rings.  
Roughly speaking, a ladder is a subset $Y$ of an $m \times n$ matrix $X$ of indeterminates that (possibly) excludes matrix entries from the top left and/or bottom right, as in the 
examples depicted below.

\begin{align*}
\begin{smallmatrix}
\displaystyle X_{11}&\displaystyle X_{12}&\displaystyle X_{13}&\displaystyle X_{14}&\displaystyle X_{15} \\
\displaystyle X_{21}&\displaystyle X_{22}&\displaystyle X_{23}&\displaystyle X_{24}&\displaystyle X_{25} \\
\displaystyle X_{31}&\displaystyle X_{32}&\displaystyle X_{33}&\displaystyle X_{34}&\displaystyle X_{35} \\
\displaystyle X_{41}&\displaystyle X_{42}&\displaystyle X_{43}&\displaystyle X_{44} \\
\displaystyle X_{51}&\displaystyle X_{52}&\displaystyle X_{53}
\end{smallmatrix}
&&&&
\begin{smallmatrix}
&\displaystyle X_{12}&\displaystyle X_{13} \\
&\displaystyle X_{22}&\displaystyle X_{23} \\
\displaystyle X_{31}&\displaystyle X_{32}&\displaystyle X_{33} \\
\displaystyle X_{41}&\displaystyle X_{42} \\
\displaystyle X_{51}&\displaystyle X_{52}
\end{smallmatrix}
\end{align*}\label{ladders:OT}

\hskip.7cm\text{$O$:  one-sided \cite[Example 4.10]{Co}} \hskip1.1cm\text{$T$: ladder with coincidental corner} 
\bigskip

The associated ladder determinantal ring of $t$-minors
is $R_t(Y) = \sfk[Y]/I_t(Y)$, where $I_t(Y)$ is the ideal generated by the $t \times t$ minors of $X$ lying entirely in $Y$.  See the paper of Aldo Conca~\cite{Co} and our Section~\ref{sec180114a} 
for background on these rings, including information on their divisor
class groups that is crucial for our work.

The main results of Sections~\ref{sec180114b} and~\ref{sec180114c}
of the current paper are as follows. 
They show that many ladder determinantal rings have only trivial semidualizing modules.
See, however, part~II of this work \cite{SWSeSp2} for the study of ladder determinantal rings with non-trivial semidualizing modules; Example \ref{multipleladders}(3) contains a sample computation.

\

\noindent
{\bf One-Sided Ladder Theorem }
(Theorem~\ref{thm:onesidedpathconnected}).
{\it Let $Y$ be a one-sided ladder\footnote{In Definition~\ref{defn:onetwosided}, we require all of our one- and two-sided ladders to be path-connected.}.  The ring $R_t(Y)$ has only trivial semidualizing modules, i.e., $|\s_0(R_t(Y))| \leq 2$.}

\

For two-sided ladders, we focus specifically on the
$2 \times 2$ case.  

\

\noindent
{\bf 
Two-Sided Ladder Theorem ($t = 2$, no coincidental corners)}  (Theorem~\ref{thm:mixed}). 
{\it Let $Y$ be a 2-connected ladder such that no lower inside corner and upper inside corner coincide.
Then the ring $R_2(Y)$ has only trivial semidualizing modules, i.e., $|\s_0(R_2(Y))| \leq 2$.}

\


\section{Background}
\label{sec180114a}


\subsection*{Divisor Class Groups}

For a normal domain $R$, the isomorphism class of an $R$-module $M$ is denoted $[M]$,
and the set of isomorphism classes of rank-1 reflexive modules is the
\textbf{divisor class group} of $R$, denoted $\Cl(R)$.
This is an abelian group under the operations $[M]+[N]=[(M\otimes_RN)^{**}]$, where $(-)^*=\Hom -R$, 
and $[M]-[N]=[\Hom NM]$, with additive identity $[R]$.  Equivalently, $\Cl(R)$ is the set of isomorphism classes of height-1 reflexive ideals.

\subsection*{Semidualizing Modules/Ideals}
Recall the definition of the semidualizing property and the notation $\s_0(R)$ from the introduction of this paper.  By~\cite[Proposition~3.4]{SSW}, if $R$ is a normal domain, then each semidualizing $R$-module is
reflexive of rank 1, so there is an inclusion $\s_0(R)\subseteq\Cl(R)$.  A {\bf semidualizing ideal} is an ideal of the ring $R$ that is semidualizing as an $R$-module.

\begin{rmk} For our purposes, it is important to note that the property of being semidualizing is preserved under localization, since the defining conditions are preserved by flat base change.  
\end{rmk}

\begin{fact} \cite[Proposition 3.3]{SSW} \label{fact:multisiso}
Let $\fa$ and $\fb$ be semidualizing ideals such that $\fa \otimes \fb$
is semidualizing. The natural multiplication map $\fa \otimes_R \fb \to
\fa \fb$ is an isomorphism.
\end{fact}

A \textbf{dualizing $R$-module} is a semidualizing $R$-module of finite injective dimension.
The ring $R$ has a dualizing module if and only if it is Cohen-Macaulay and a homomorphic image of a Gorenstein ring with finite Krull dimension; 
see~\cite{foxby:gmarm,reiten:ctsgm,sharp:gmccmlr}.
Thus, a dualizing module is  a canonical module over a Cohen-Macaulay ring. To say that a ring $R$ admits only {\bf trivial semidualizing modules} means $\s_0(R) = \{[R], [\omega_R] \}$ if $R$ has a dualizing module $\omega_R$, and it means that $\s_0(R)=\{[R]\}$ if $R$ does not have a dualizing module.

\begin{fact}\label{fact:tensorgivesomega} If 
$R$ is Cohen-Macaulay with a dualizing module $\omega_R$, and if $C$ is a semidualizing $R$-module, then
$\Hom C{\omega_R}$ is semidualizing.  Moreover, 
the natural evaluation map
$\gamma\colon\Otimes{\Hom C{\omega_R}}C\to\omega_R$ given by $\gamma(\varphi\otimes c)=\varphi(c)$ is an isomorphism, 
and $\Tor i{\Hom C{\omega_R}}C =0$ for all $i\geq 1$; see~\cite[Theorem~2.11, Proposition~4.4, and Observation~4.10]{christensen:scatac}.
If, in addition, $R$ is a normal domain and  $C\neq\omega_R$ are height-1 reflexive ideals,
then $\Hom C{\omega_R}$ is naturally isomorphic to a height-1 reflexive ideal $C'$, and we have
$\omega_R\xra[\gamma^{-1}]{\cong}\Otimes C{C'} \xra{\cong}CC'$, where the second isomorphism is the multiplication map from Fact~\ref{fact:multisiso}.
\end{fact}


\subsection*{Ladder Determinantal Rings}


We will recall the terminology and results in \cite{Co} and \cite{CoGor}
and also introduce some new terminology.

Let $X=(X_{ij})$ be an $m\times n$ matrix of indeterminates.
A \textbf{ladder} in $X$ is a subset $Y$ satisfying the following property:
if $X_{ij},X_{pq}\in Y$ satisfy $i\leq p$ and $j\leq q$, then $X_{iq},X_{pj}\in Y$.  Recall that $R_t(Y) = \sfk[Y]/I_t(Y)$ is the associated \textbf{ladder determinantal ring}, where $I_t(Y)$ is the ideal generated by the $t \times t$ minors of $X$ lying entirely in $Y$.  
As in~\cite[p.~121(b)]{Co}, to avoid trivialities, we assume without loss of generality that $X_{m1},X_{1n}\in Y$ and furthermore
that each row of $X$ contains an element of $Y$, as does  each column of $X$.
One such ladder is as follows.
\begin{equation}\label{eq180304a}\tag{L}
\begin{smallmatrix}
&\displaystyle X_{12}&\displaystyle X_{13}&\displaystyle X_{14}&\displaystyle X_{15} \\
&\displaystyle X_{22}&\displaystyle X_{23}&\displaystyle X_{24}&\displaystyle X_{25} \\
\displaystyle X_{31}&\displaystyle X_{32}&\displaystyle X_{33} \\
\displaystyle X_{41}&\displaystyle X_{42}&\displaystyle X_{43} \\
\displaystyle X_{51}&\displaystyle X_{52}
\end{smallmatrix}
\end{equation}
Herzog and Trung~\cite[Corollary~4.10]{MR1185786} show that the ring $R_t(Y)$ is Cohen-Macaulay,
and it is a normal domain by~\cite[Proposition~3.3]{Co}.
For each $X_{ij}\in Y$, we let $x_{ij}$ denote its residue in $R_t(Y)$.

The \textbf{lower inside corners} of $Y$ are the points $(a,b)$ with $X_{ab}, X_{a-1 b}, X_{a b-1} \in Y$, but $X_{a-1 b-1}  \in X \ssm Y$; these are denoted $X_{a_i b_i}$, or simply $(a_i, b_i)$, with $1 <a_1 <\cdots<a_h <m$. For notational convenience, we also set $(a_0,b_0)=(1,n)$ and $(a_{h+1},b_{h+1})=(m,1)$.  Likewise, the \textbf{upper inside corners} of a ladder $Y$ are the points $(c,d)$ such that $X_{cd}, X_{c+1 d}, X_{c d+1} \in Y$, but $X_{c+1 d+1}  \in X \ssm Y$; these are denoted $X_{c_j d_j}$, or simply $(c_j, d_j)$, with $1 <c_1 <\cdots<c_k <m$. 
The ladder $Y$ has \textbf{coincidental corners} if $(a_i,b_i)=(c_j,d_j)$ for some $i,j$.
For notational convenience, we also set $(c_0,d_0)=(1,n)$ and $(c_{k+1},d_{k+1})=(m,1)$.

For instance, the ladder~\eqref{eq180304a} above has $h=1$ and $k=2$, with $(a_0,b_0) = (1,5) = (c_0,d_0)$ and $(a_2,b_2) = (5,1) = (c_3,d_3)$, and the variables at inside corners are boxed in the next display.
\begin{align*}
&\begin{smallmatrix}
&\displaystyle X_{12}&\displaystyle X_{13}&\displaystyle X_{14}&\displaystyle X_{15} \\
&\displaystyle X_{22}&\displaystyle X_{23}&\displaystyle X_{24}&\displaystyle X_{25} \\
\displaystyle X_{31}&\fbox{$\displaystyle X_{32}$}&\displaystyle X_{33} \\
\displaystyle X_{41}&\displaystyle X_{42}&\displaystyle X_{43} \\
\displaystyle X_{51}&\displaystyle X_{52}
\end{smallmatrix}
&&\begin{smallmatrix}
&\displaystyle X_{12}&\displaystyle X_{13}&\displaystyle X_{14}&\displaystyle X_{15} \\
&\displaystyle X_{22}&\fbox{$\displaystyle X_{23}$}&\displaystyle X_{24}&\displaystyle X_{25} \\
\displaystyle X_{31}&\displaystyle X_{32}&\displaystyle X_{33} \\
\displaystyle X_{41}&\fbox{$\displaystyle X_{42}$}&\displaystyle X_{43} \\
\displaystyle X_{51}&\displaystyle X_{52}
\end{smallmatrix}
\\
&\hspace{0.5cm}\text{lower inside corners}
&&\hspace{0.5cm}\text{upper inside corners}
\end{align*}
One point of identifying the inside corners is to describe $\Cl(R_t(Y))$.  In particular, for $t = 2$, 
the following ideals of $R_2(Y)$ are height-1 primes by~\cite[Proposition~2.1 and Corollary~2.3]{Co}:
\begin{align*}
\p_j
&=(x_{pq}\in R_2(Y)\mid\text{$p\leq c_j$ and $q\leq d_j$})&j&=1,\ldots,k
\\
\q_i
&=(x_{a_{i-1}q}\in R_2(Y))&i&=1,\ldots,h+1
\\
\q_i'
&=(x_{pb_{i}}\in R_2(Y))&i&=1,\ldots,h+1.
\end{align*} \label{defn:pq}%

\begin{fact} \label{omega} The following facts were established in \cite{Co}.
\begin{enumerate}
\item $\Cl(R_2(Y))$ is a free abelian group of rank $h+k+1$ with basis $[\q_1]$, \ldots, $[\q_{h+1}]$,
$[\p_1]$, \ldots, $[\p_k]$, \cite[Corollary 2.3]{Co}. 
\item With $t=2$, set $\lambda_i = a_i + b_i - a_{i-1} - b_{i-1}$ for all $i = 1, \dots, h+1$ and $\delta_j = a_{i_j} + b_{i_j}-c_j-d_j$ for all $j = 1, \dots, k$, where $i_j = \min\{i :  a_i > c_j \}$. Then the canonical class is $[\omega_R] = \sum_{i=1}^{h+1} \lambda_i [\q_i] + \sum_{j=1}^k \delta_j[\p_j]$ by~\cite[Proposition~2.4]{Co}. 
\item The relations between the classes of the ideals $\q_i$, $\q_i'$, $\p_j$, described in the proof of \cite[Corollary 2.3(i)]{Co}, are as follows. For all $i = 1, \dots, h+1$ if $I_i = \{j : \text{$1 \leq j \leq k$, $a_{i-1}\leq c_j$, and $b_i \leq d_j$} \}$, where $I_i$ may be empty, then $[\q_i] + [\q_i'] + \sum_{j \in I_i} [\p_j] = 0$.
\end{enumerate}
\end{fact}

For the specific  ladder~\eqref{eq180304a} above, we have
\begin{align*}
\p_1
&=(x_{12},x_{13},x_{22},x_{23}) & & & \delta_1 = 3+2-2-3 \\
\p_2
&=(x_{12},x_{22},x_{31},x_{32},x_{41},x_{42}) & & & \delta_2 = 5+1-4-2 \\
\q_1
&=(x_{12},x_{13},x_{14},x_{15}) & & & \lambda_1 = 3+2-1-5 \\
\q_2
&=(x_{31},x_{32},x_{33}) & & & \lambda_2 = 5+1-3-2 \\
\q_1'
&=(x_{12},x_{22},x_{32},x_{42},x_{52}) \\
\q_2'
&=(x_{31},x_{41},x_{51}) \\
\Cl(R_2(L))&\cong\bbz^4\cong \bbz[\p_1]\oplus\bbz[\p_2]\oplus\bbz[\q_1]\oplus\bbz[\q_2] \\
[\omega_{R_2(L)}]&=[\q_2] - [\q_1].
\end{align*}

A ladder $Y$ is {\bf $\pmb t$-disconnected} \cite[page 457]{CoGor} if there exist two subladders $\varnothing \neq Z_1, Z_2 \subseteq Y$ such that $Z_1 \cap Z_2 = \varnothing$, $Z_1 \cup Z_2 = Y$, and every $t$-minor of $Y$ is contained in $Z_1$ or $Z_2$. 
In this case, we say that $Z_1,Z_2$ form a {\bf $\pmb t$-disconnection} of $Y$.
A ladder
$Y$ is {\bf $\pmb t$-connected} if it is not $t$-disconnected.
(For instance, the ladder~\eqref{eq180304a} above is 2-connected and is vacuously $t$-disconnected 
for each $t\geq 3$ since it has no 3-minors.  The one-sided ladder $O$ from the introduction is 2- and 3-connected, but not 4-connected.) 
A \textbf{block submatrix} of $Y$ is a rectangular subladder, that is, a subset of $Y$
consisting of all the $X_{pq}$ with $u \leq p \leq v$ and $r \leq q \leq s$ for some $u,v,r,s$.

We use the following natural definitions below to compute semidualizing modules in some one-sided cases where $Y$
is not $t$-connected.

\begin{defn}
  A \textbf{path} in a ladder $Y$ is a nonempty (but possibly one-element)
  list of variables
  $X_{i_0j_0}, X_{i_1j_1}, \dots, X_{i_{\ell}j_{\ell}}$
  in $Y$, such that for all $0\leq u < \ell$, either
  \begin{enumerate}
    \item $i_u=i_{u+1}$ and $|j_u-j_{u+1}|=1$, or
    \item $j_u=j_{u+1}$ and $|i_u-i_{u+1}|=1$.
  \end{enumerate}
  In such a case, we say that there is a path \textbf{from $X_{i_0j_0}$ to $X_{i_{\ell}j_{\ell}}$},
  or \textbf{between $X_{i_0j_0}$ and $X_{i_{\ell}j_{\ell}}$}. 
  We write
  $X_{i_0j_0}\sim X_{i_{\ell}j_{\ell}}$ to denote that there is a path from
  $X_{i_0j_0}$ to $X_{i_{\ell}j_{\ell}}$.  

  The \textbf{path-components} of a ladder $Y$ are the equivalence classes of
  $\sim$. 
  A ladder $Y$ is \textbf{path-connected} if there is a path between any
  two variables in $Y$, or equivalently, $Y$ has only one path-component.
  A ladder is \textbf{path-disconnected} if it is not path-connected.
\end{defn}

\begin{lem}
  Every path-component of a ladder is also a ladder.
\end{lem}

\begin{proof}
   Let $Y$ be a ladder and $Y_1$ a path-component of $Y$. Let $X_{ij},X_{pq} \in Y_1$
  with $i \leq p$ and $j \leq q$. Let $Z$ be the block submatrix of $X$ with corners
  $X_{ij}$ and $X_{pq}$. By the defining condition for $Y$ being a ladder, we have
  $Z \subseteq Y$, so there are paths in $Y$ between all of the variables
  $X_{ij},X_{pq},X_{iq},X_{pj}$. Hence $X_{iq},X_{pj} \in Y_1$, and $Y_1$
  is a subladder of $Y$.
\end{proof}

The ladder $Y$ in Example~\ref{ex180710a} shows that the converse of the next result fails with $t=3$;
examples for other $t$-values are similarly easy to construct.

\begin{lem} \label{lem:connectivity}
  A $t$-connected ladder is path-connected for any $t>0$.
\end{lem}

\begin{proof}
  The only 1-connected ladder is the one consisting of one variable, so the result is easy in this case. 
  Thus, we let $t>1$ and assume that $Y$ is $t$-connected. By way of contradiction, suppose that
  $Y$ were path-disconnected.
  Let $Y_1$ be a path-component of $Y$ and $Y_2=Y\ssm Y_1$. Then
  $Y_1,Y_2$ are ladders. Since $Y$ is $t$-connected, 
  there is a $t$-minor given by a set of variables $M$ of $Y$ such that
  $M \cap Y_1 \neq \varnothing$ and
  $M \cap Y_2 \neq \varnothing$. Let $Z$ be the smallest block submatrix of $X$ that
  contains $M$. Then $Z \subseteq Y$, so $Y_1 \cup Z$ is path-connected and properly
  contains $Y_1$, a contradiction. Therefore $Y$ is path-connected.
\end{proof}

\begin{defn} \label{defn:onetwosided}
A ladder $Y$ is {\bf one-sided} if it is path-connected and $h=0$ or $k=0$, i.e.\ it has no lower inside corners or no upper inside corners.
When this is the case, we usually assume that $h=0$, by symmetry.
A ladder $Y$ is {\bf two-sided} if it is path-connected and $h,k>0$.
\end{defn}
Since all ladders in \cite{CoGor} are assumed to be $t$-connected, our definitions of
one-sided and two-sided ladders are compatible with those in \cite[pp.~457, 458]{CoGor}
 by Lemma~\ref{lem:connectivity}.


\section{One-Sided Ladders}
\label{sec180114b}


In this section, we prove the One-Sided Ladder Theorem from the introduction.
Note that the results of this section will be applied in the next section.
In particular, Lemmas~\ref{lem:purepower}--\ref{lem:intersection} apply to arbitrary 2-connected ladders (one- or two-sided).

We recall that the ideals $\q_i,\q'_i,\p_j$ were defined on page~\pageref{defn:pq}.

\begin{lem}\label{lem:purepower}
  Let $Y$ be a 2-connected ladder and set $R=R_2(Y)=\sfk[Y]/I_2(Y)$. Then for all $1 \leq i \leq h+1$,
  $1 \leq j \leq k$ and $e \in \bbn$, we have $\q_i^{(e)}=\q_i^e$,
  $(\q'_i)^{(e)}=(\q'_i)^e$ and $\p_j^{(e)}=\p_j^e$ in $R$.
\end{lem}

\begin{proof}
  Let $J$ denote any fixed ideal $\q_i$, $\q'_i$ or $\p_j$. For any polynomial
  $f \in \sfk[Y]$, we let $\bar{f}$ denote its residue class in $R$. We define a grading
  on $\sfk[Y]$ by letting $\deg(X_{ij})=1$ if $x_{ij} \in J$ and $\deg(X_{ij})=0$
  otherwise. We note that the generators $X_{i_1j_1}X_{i_2j_2}-X_{i_1j_2}X_{i_2j_1}
  \in I_2(Y)$, where $i_1 \leq i_2$ and $j_1 \leq j_2$, are homogeneous binomials of degree
  0, 1 or 2. Hence $R$ inherits the same grading from $\sfk[Y]$.

  Given a polynomial $f \in \sfk[Y]$, let us write $f=f_r+f_{r+1}+\dots+f_t$,
  where $f_s$ is homogeneous of degree $s$ for $r \leq s \leq t$. We note that
  whenever $e \in \bbn$ and $\bar{f_r}\neq 0$, we have $\bar{f}\in J^e$ if and only if
  $r \geq e$.

  Now fix $e \in \bbn$ and let $f \in \sfk[Y]$. Let $f=f_r+f_{r+1}+\dots+f_t$
  with $\bar{f_r}\neq 0$. Suppose that $\bar{f} \in J^{(e)}$. Then there is
  $g \in \sfk[Y]$ with $g=g_0+g_1+\dots+g_s$ and $\bar{g_0}\neq 0$,
  such that $\bar{f}\bar{g} \in J^e$. Since $R$ is a domain, we have
  $\bar{f_r}\bar{g_0} \neq 0$. Since $\deg(f_rg_0)=r$ and $\bar{f}\bar{g}
  \in J^e$, we have $r \geq e$. Hence $\bar f \in J^e$. Therefore $J^{(e)}=J^e$.
\end{proof}

The $\q_i$, $\q'_i$, $\p_j$ are residues in $R_2(Y)$ of ideals in $\sfk[Y]$ \cite[\S 2]{Co}.  In particular, $\q_i$ is the residue class of $Q_i = (X_{a_{i-1}, j} : {\text{ for all }} j {\text{ such that }} X_{a_{i-1}, j} \in Y) + I_2(Y)$.

\begin{lem} \label{lem:intersection}
  Let $Y$ be a 2-connected ladder. Let $J$ be any one of the ideals
  $Q_i^e$, $(Q'_i)^e$ or $P_j^e$, where $e \geq 1$. Let $M$ be any
  monomial ideal in $\sfk[Y]$. Then $\bar{J} \cap \bar{M} \subseteq
  R = R_2(Y)$ is generated by the least common multiples of the
  monomial generators of $J$ and $M$. In other words, $\bar{J} \cap
  \bar{M} = \overline{J \cap M}$.
\end{lem}

\begin{proof}
Let $f\in \sfk[Y]$ be such that $\bar f\in\bar J\cap\bar M$. By collecting terms from $I_2(Y)$, we may assume that
  $f=f_J + f_I = f_M$, where $f_J \in J$, $f_I \in I_2(Y)$ and $f_M \in M$.
  Consider a term $t_J$ that appears in $f_J$. Suppose that
  there is cancellation between $t_J$ and a term
  $rm(X_{i_1,j_1}X_{i_2,j_2}-X_{i_1,j_2}X_{i_2,j_1})$ in $f_I$,
  where $r \in \sfk$, $m$ is a monomial, $i_1<i_2$ and $j_1<j_2$.
  Then considering the grading in Lemma~\ref{lem:purepower}, since
  $X_{i_1,j_1}X_{i_2,j_2}-X_{i_1,j_2}X_{i_2,j_1}$ is homogeneous
  of degree 0, 1 or 2, we have $rm(X_{i_1,j_1}X_{i_2,j_2}
  -X_{i_1,j_2}X_{i_2,j_1}) \in J$. So by rearranging terms, we may
  assume that no monomial in $f_I$ belongs to $J$. Now if a term
  $t_J$ appears in $f_J$, then it does not cancel with any term in
  $f_I$, so $t_J$ appears in $f_M$ since
  $f_J + f_I = f_M$. Therefore $t_J$ is a common multiple of one
  monomial generator in $J$ and one in $M$. Finally, we recall that the
  intersection of two monomial ideals is generated by the least
  common multiples of their respective monomial generators.
\end{proof}
The proofs below will involve  new ladders obtained from a given ladder $Y$.  

\begin{notn} \label{notn:ladder}
  When we are considering a ladder $Y$ and would like to discuss a new related
  ladder $\ti{Y}$, we denote the corners of $\ti Y$ as $(\ti{a}_i,\ti{b}_i)$ and $(\ti{c}_j,\ti{d}_j)$.
 The notation $\ti{R}$ will  denote the associated ladder determinantal ring $R_2(\ti{Y})$, with prime
  ideals such as $\ti{\p}_1$, $\ti{\p}_2$, $\ti{\q}_1$, $\ti{\q}_2$, etc.
  Similar protocols apply for ladders $\check Y$.
\end{notn}

\begin{thm}  \label{thm:onesided}
Let $R = R_2(Y)$ be a one-sided 2-connected ladder determinantal ring.  Then $|\mathfrak S_0(R)| \leq 2$.
\end{thm}

\begin{proof}
  Since $Y$ is one-sided, we assume without loss of generality that $h=0$. 
  The proof is by induction on the number $k$ of (upper) inside corners of $Y$.
  The case $k=0$ is given by \cite[Theorem 4.2]{SSW}, therefore, let $k>0$. As per Fact \ref{omega}, we have $[\omega_R] = \sum_{j=0}^k
  \gd_j [\p_j]$, where $\p_0=\q_1$ and $\gd_j = a_1+b_1-c_j-d_j$ for all
  $0 \leq j \leq k$. 
  
  Consider the ladder $\ti{Y}$ obtained by deleting rows $c_0,c_0+1,\dots,c_1-1$
  and columns $d_1+1,d_1+2,\dots,d_0$ of $Y$. Then $[\omega_{\ti{R}}]=
  \sum_{j=0}^{k-1} \gd_{j+1} [\ti{\p}_j]$. Let us invert $x_{c_1d_0}$ in $R$ and
  let $\ti\vf$ be the composition of the following natural surjections:
  \[
    \Cl(R) \to \Cl(R_{x_{c_1d_0}}) \xrightarrow{\cong} \Cl(\ti{R}).
  \]
  The maps here come from the flat maps
  $$R\to R_{x_{c_1d_0}}\xra\cong\ti R[X_{c_0d_0},\ldots,X_{c_1d_0},\ldots,X_{c_1,d_1+1}]_{X_{c_1d_0}}\from \ti R.$$
 In particular, the  maps on divisor class groups respect semidualizing modules by~\cite[Lemma 3.10(a)]{SSW}.
  
  We have $\ti\vf([\p_0]) = 0$ and $\ti\vf([\p_i]) = [\ti{\p}_{i-1}]$ for
  $1 \leq j \leq k$, hence $\ker \ti\vf = \bbz [\p_0]$. By our induction hypothesis, the only
  semidualizing modules of $\ti{R}$ are $[\ti{R}]$ and $[\omega_{\ti{R}}]$.
  Since the localization of a semidualizing module is also a semidualizing
  module, the only possible semidualizing modules of $R$ are in
  $\ti\vf^{-1}([\ti{R}]) = \bbz[\p_0]$ or $\ti\vf^{-1}([\omega_{\ti{R}}]) = \bbz[\p_0]
  + \sum_{j=1}^k \gd_j [\p_j]$. Let us write the possible semidualizing modules
  of $R$ as $[N_1]=r[\p_0]$ and $[N_2]=s[\p_0]+ \sum_{j=1}^k \gd_j [\p_j]$,
  where $r,s \in \bbz$.
  
  Next, we invert $x_{c_{k+1}d_k}$ and obtain $\check{Y}$ by deleting rows
  $c_k+1,c_k+2,\dots,c_{k+1}$ and columns $d_{k+1},d_{k+1}+1,\dots,d_k-1$
  of $Y$. Then $\Cl(\check{R})$ is generated by the basis elements
  $[\check{\p}_0],[\check{\p}_1],\dots,[\check{\p}_{k-1}]$, and $[\omega_{\check{R}}]
  = \sum_{j=0}^{k-1} (c_k+d_k-c_j-d_j)[\check{\p}_j]$.
  
  Under the natural map
  $\check\vf \colon \Cl(R) \to \Cl(\check{R})$, we have $\check\vf([\p_j])=[\check{\p}_j]$ for
  all $0 \leq j \leq k-1$ and $\check\vf([\p_k])=[\check{\q}'_1]=-\sum_{j=0}^{k-1}
  [\check{\p}_j]$ by \cite[Proposition 2.1]{Co}. By our induction hypothesis, the only classes of
  semidualizing modules of $\check{R}$ are $[\check{R}]$ and $[\omega_{\check{R}}]$.
  By assumption, the classes of semidualizing modules of $R$ are of the form $[N_1]$ and $[N_2]$,
  so we must have $\check\vf([N_i])=0$ or $[\omega_{\check{R}}]$ for $i=1,2$.
  
  If $0=\check\vf([N_1])=r[\check{\p}_0]$, then $[N_1]=0$. Similarly, if $[\omega_{\check{R}}]=\check\vf([N_2])=
  (s-\gd_k)[\check{\p}_0]+\sum_{j=1}^{k-1}(\gd_j-\gd_k)[\check{\p}_j]$, then $[N_2]=[\omega_R]$.
  
  If $[\omega_{\check{R}}]=\check\vf([N_1])=r[\check{\p}_0]$, then $r=c_k+d_k-c_0-d_0
  =\gd_0-\gd_k$, and $c_1+d_1=c_2+d_2=\dots=c_k+d_k$; i.e., all inside corners lie on the same
  \lq\lq antidiagonal\rq\rq, by which we simply mean the same line (and do not require that the matrix be square).  In this case, $[N_1]$
  gives us the possible nontrivial semidualizing module $[M_1]=(\gd_0-\gd_k)
  [\p_0]=(\gd_0-\gd_1)[\p_0]$.
 Hence $M_2=\Hom{M_1}{\omega_R}$ is also semidualizing for $R$ with $[M_2]=\delta_1\sum_{j=0}^k[\p_j]$. 
  
  On the other hand, if $\check\vf([N_2])=(s-\gd_k)[\check{\p}_0]+\sum_{j=1}^{k-1}(\gd_j-\gd_k)[\check{\p}_j]$, 
  then $s=\gd_k$ and $\gd_j=\gd_k$ for all $1 \leq j \leq k-1$.
  In this case, $[N_2]$ gives us the possible nontrivial semidualizing module
  $[M_2]=\gd_1\sum_{j=0}^k[\p_j]$, where all inside corners lie on the same
  antidiagonal, so $[M_2]=[\omega_R]-[M_1]$.
 Hence $M_1=\Hom{M_2}{\omega_R}$ is also semidualizing for $R$ with $[M_1]=(\delta_0-\delta_1)[\p_0]$. 
  
  Let us write $\zeta=\gd_0-\gd_1$ and $\gd=\gd_1$, so that $[M_1]=\zeta[\p_0]$
  and $[M_2]=\gd\sum_{j=0}^k[\p_j]$. Since $[M_1]+[M_2]=[\omega_R]$, 
  it suffices to show that we get a contradiction if $\zeta\gd \neq 0$.

  \begin{case} \label{case:notiso}
  $\zeta,\gd>0$. Using Lemmas~\ref{lem:purepower} and~\ref{lem:intersection}, one can check that
  \begin{align*}
    [M_1] &= [\p_0^{\zeta}] = [(x_{a_0d_{k+1}}, x_{a_0d_{k+1}+1},\dots,
    x_{a_0d_0})^{\zeta}] \text{ and}\\
    [M_2] &= 
    [\cap_{j=0}^k \p_j^{\gd}]
    = [(x_{a_0d_{k+1}}, x_{a_0d_{k+1}+1},\dots, x_{a_0d_k})^{\gd}].
  \end{align*}
  Let us identify $M_1$ with the ideal $(x_{a_0d_{k+1}}, x_{a_0d_{k+1}+1},
  \dots, x_{a_0d_0})^{\zeta}$ above and $M_2$ with $(x_{a_0d_{k+1}},
  x_{a_0d_{k+1}+1},\dots,x_{a_0d_k})^{\gd}$. Under the multiplication map
  $\mu \colon M_1 \otimes_R M_2 \to M_1M_2$, we have
  \[
    \mu(x_{a_0d_{k+1}}^{\zeta} \otimes x_{a_0d_{k+1}}^{\gd-1} x_{a_0d_{k+1}+1})
    = \mu(x_{a_0d_{k+1}}^{\zeta-1}x_{a_0d_{k+1}+1} \otimes
    x_{a_0d_{k+1}}^{\gd}).
  \]
  So $\mu$ is not injective. If $M_1,M_2$ are semidualizing modules of $R$,
  then we get a contradiction by Fact~\ref{fact:multisiso}.
  \end{case}
  
  \begin{case}
  $\zeta,\gd<0$. By \cite[Corollary 2.3(i)]{Co}, we have
  $\sum_{i=0}^{k+1} [\p_i] = 0$, where $[\p_{k+1}]=[\q'_1]$, giving
  \begin{align*}
    [M_1] &= \zeta[\p_0] = -\zeta\sum_{i=1}^{k+1}[\p_i] \quad \text{and}\\
    [M_2] &= \gd \sum_{i=0}^k [\p_i] = -\gd[\p_{k+1}].
  \end{align*}
  We may then use Case~\ref{case:notiso} by symmetry.
  \end{case}
  
  \begin{case} \label{case:mingen}
  $\zeta>0$, $\gd<0$. Again \cite[Corollary 2.3(i)]{Co} and Lemma~\ref{lem:purepower}  give
  \begin{align*}
    [M_1] &= [\p_0^{\zeta}]
      = [(x_{a_0b_1}, x_{a_0b_1+1},\dots, x_{a_0b_0})^{\zeta}],\\
    [M_2] &= \left|\gd\right| [\p_{k+1}] = [\p_{k+1}^{|\gd|}]
      = [(x_{a_0b_1},x_{a_0+1b_1},\dots, x_{a_1b_1})^{|\gd|}], \text{ and}\\
    [\omega_R] &= [M_1] + [M_2] = [\p_0^{\zeta} \cap \p_{k+1}^{|\gd|}].
  \end{align*}
  Let us identify $M_1,M_2,\omega_R$ with the ideals on the right, as
  in Case~\ref{case:notiso}.
  
Now we use the fact that if $M_2$ is semidualizing, then so is
$\Hom{M_2}{\omega_R}$, and we have  isomorphisms
$$M_1 \otimes_R M_2\cong \Hom{M_2}{\omega_R} \otimes_R M_2\xra\cong \omega_R,$$
where the second map is given by evaluation.
In particular, it follows that the modules $M_1 \otimes_R M_2$ and $\omega_R$
have minimal generating sets of the same size.

Let $\mingen(\p_0^{\zeta})$ denote the set of all monomials
$m_1=x_{a_0b_1}^{p_0} x_{a_0b_1+1}^{p_1}\cdots x_{a_0b_0}^{p_{b_0-b_1}}$ in $R$ of degree $\zeta$.
This is a minimal generating set for $\p_0^{\zeta}$.
Similarly, the set $\mingen(\p_{k+1}^{|\gd|})$ of all monomials $m_2=x_{a_0b_1}^{q_{0}}x_{a_0+1b_1}^{q_{1}}\cdots x_{a_1b_1}^{q_{a_1-a_0}}$
in $R$ of degree $|\gd|$ is a minimal generating set for $\p_{k+1}^{|\gd|}$.
By abuse of notation, we write 
\newcommand{\lcm}{\operatorname{lcm}}%
$\lcm(m_1,m_2)$
for the monomial
$$
\lcm(m_1,m_2)=x_{a_0b_1}^{\max(p_0,q_0)} x_{a_0b_1+1}^{p_1}\cdots x_{a_0b_0}^{p_{b_0-b_1}}x_{a_0+1b_1}^{q_{1}}\cdots x_{a_1b_1}^{q_{a_1-a_0}}.
$$
Then Lemma~\ref{lem:intersection} shows that the function 
$$\lcm\colon\mingen(\p_0^{\zeta})\times\mingen(\p_{k+1}^{|\gd|})\to\p_0^{\zeta} \cap \p_{k+1}^{|\gd|}$$
has its image equal to a generating set for $\p_0^{\zeta} \cap \p_{k+1}^{|\gd|}$.
The previous paragraph shows that each minimal generating set for $\p_0^{\zeta} \cap \p_{k+1}^{|\gd|}$ must have the same size
as $\mingen(\p_0^{\zeta})\times\mingen(\p_{k+1}^{|\gd|})$, so the lcm map here must be injective with  image equal
to a minimal generating set for $\p_0^{\zeta} \cap \p_{k+1}^{|\gd|}$. 
However, this is not the case because the following computation exhibits an lcm in $\m(\p_0^{\zeta} \cap \p_{k+1}^{|\gd|})$
where $\m$ is the homogeneous maximal ideal of $R$:
\begin{align*}
\lcm(x_{a_0b_1}^{\zeta-1} x_{a_0b_1+1},x_{a_0b_1}^{|\gd|-1}x_{a_0+1b_1})
&=x_{a_0b_1}^{\max(\zeta,|\gd|)-1} x_{a_0b_1+1}x_{a_0+1b_1}\\
&=x_{a_0b_1}^{\max(\zeta,|\gd|)} x_{a_0+1,b_1+1}\\
&=\lcm(x_{a_0b_1}^{\zeta},x_{a_0b_1}^{|\gd|})x_{a_0+1,b_1+1}.
\end{align*}
Hence this lcm cannot be part of a minimal generating set for $\p_0^{\zeta} \cap \p_{k+1}^{|\gd|}$.

  \end{case}
  
  \begin{case}
  $\zeta<0$, $\gd>0$. Then $[M_1] = \zeta[\p_0] = \left| \zeta \right|
  \sum_{j=1}^{k+1} [\p_j]$. Since $[M_2] = \gd\sum_{j=0}^k [\p_j]$, we may
  assume that $\gd \geq |\zeta|$ by symmetry, so that $\gd_1 > \gd_0 \geq 0$.
  Then using Lemmas~\ref{lem:purepower} and~\ref{lem:intersection} we have
  \begin{align*}
    [M_1] &= \big[\cap_{j=1}^{k+1} \p_j^{|\zeta|} \,\big]
    = [(x_{c_0d_{k+1}},x_{c_0+1d_{k+1}},\dots,x_{c_1d_{k+1}})^{|\zeta|}],\\
    [M_2] &= \big[\cap_{j=0}^k \p_j^{\gd} \,\big]
    = [(x_{c_0d_{k+1}},x_{c_0d_{k+1}+1},\dots,x_{c_0d_k})^{\gd}],
    \text{ and}\\
    [\omega_R] &= \gd_0[\p_0] + \gd_1\sum_{j=1}^k [\p_j]
    = \left[\p_0^{\gd_0} \cap \bigcap_{j=1}^k \p_j^{\gd_1}\right]
    = \left[\left(\prod_{u=1}^{\gd_0} x_{c_0j_u} \prod_{v=1}^{\gd_1-\gd_0}
       x_{i_v,j_{\gd_0+v}} \right)\right],
  \end{align*}
  where $c_0 \leq i_1,i_2,\dots,i_{\gd_1-\gd_0} \leq c_1$ and
  $d_{k+1} \leq j_1,j_2,\dots,j_{\gd_1} \leq d_k$.
  We again identify $M_1,M_2,\omega_R$ with the ideals shown above.
  Here, the
  multiplication map
  \begin{align*}
    &(x_{c_0d_{k+1}},x_{c_0+1d_{k+1}},\dots,x_{c_1d_{k+1}})^{\gd_1-\gd_0}
    \otimes_R (x_{c_0d_{k+1}},x_{c_0d_{k+1}+1},\dots,x_{c_0d_k})^{\gd_1}\\
    & \to  x_{c_0d_{k+1}}^{\gd_1-\gd_0} \omega_R \cong \omega_R
  \end{align*}
  may actually give an isomorphism.
  So to get a contradiction, we will use the fact that if $M_2$ is
  semidualizing, then
  \[
    \Tor{1}{M_1}{M_2} \cong \Tor{1}{\Hom{M_2}{\omega_R}}{M_2} = 0.
  \]
  Consider a minimal free resolution of $M_2$ as follows.
  \begin{equation} \label{eqn:c2res}
    0\from M_2 \xleftarrow{\partial_0} R^{\gb_0} \xleftarrow{\partial_1} R^{\gb_1}
    \xleftarrow{\partial_2} R^{\gb_2} \leftarrow \cdots, \quad \text{where}
  \end{equation}
  \begin{align*}
    \partial_0 &= \left( \begin{matrix}
      x_{c_0d_{k+1}}^{\gd} & x_{c_0d_{k+1}}^{\gd-1}x_{c_0d_{k+1}+1}
      & \cdots
      & x_{c_0d_{k+1}}x_{c_0d_{k+1}+1}^{\gd-1} & x_{c_0d_{k+1}+1}^{\gd}
      & \cdots
    \end{matrix} \right),\\
    \partial_1 &= \left( \begin{array}{c c c c c c c c c}
      x_{c_0d_{k+1}+1} & x_{c_0+1d_{k+1}+1} & \cdots & x_{c_{k+1}d_{k+1}+1}
      & \cdots\\
      -x_{c_0d_{k+1}} & -x_{c_0+1d_{k+1}} & \cdots & -x_{c_{k+1}d_{k+1}}
      & \cdots\\
      0 & 0 & \cdots & 0 & \cdots\\
      0 & 0 & \cdots & 0\\
      \vdots & & & \vdots & \vdots\\
      0 & 0 & \cdots & 0
      & \cdots
    \end{array} \right) \text{, etc.}
  \end{align*}
  Now we truncate \eqref{eqn:c2res} and tensor with $M_1$ to get
  \[
    0 \from M_1^{\gb_0}
    \xleftarrow{\partial_1 \otimes M_1} M_1^{\gb_1}
    \xleftarrow{\partial_2 \otimes M_1} M_1^{\gb_2} \leftarrow \cdots
  \]
  We see that $\x = (x_{c_0d_{k+1}}^{|\zeta|-1}x_{c_0+1d_{k+1}},
  -x_{c_0d_{k+1}}^{|\zeta|},0,\dots,0)^T \in
  \ker(\partial_1 \otimes M_1)$. This is a minimal generator of $M_1^{\gb_1}$.
  However, since \eqref{eqn:c2res} is a
  minimal resolution, the entries of $\partial_2$ are in the homogeneous maximal
  ideal of $R$, so $\x \notin \im(\partial_2 \otimes M_1)$, giving us our
  final contradiction. \qedhere
  \end{case}
\end{proof}

\begin{cor} \label{generalonesided} Let $Y$ be a one-sided $t$-connected ladder.  Then $|\frak S_0(R_t(Y))| \leq 2$ for any $t \geq 1$.
\end{cor}

\begin{proof} We induct on $t$. If $t=1$ or if $Y$ contains no $t \times t$ minors, then $R_t(Y)$ is Gorenstein so the result is trivial.  Since the case of $t=2$ is handled above, suppose that $Y$ contains $t \times t$ minors for $t \geq 3$, and assume that for all one-sided $(t-1)$-connected ladders, the associated ladder determinantal rings of $(t-1) \times (t-1)$ minors have only trivial semidualizing modules.  Let $Z$ be the ladder obtained from $Y$ by deleting the first row and first column, which is necessarily $(t-1)$-connected.  By \cite[Proposition 4.1(2) and proof of Theorem 4.9(b)]{Co}, there is an isomorphism $\Cl(R_t(Y)) \to \Cl(R_{t-1}(Z))$.  As in the proof of Theorem~\ref{thm:onesided}, the class of any semidualizing module for $R_t(Y)$ must map to the class of a semidualizing module for $R_{t-1}(Z)$, and the result follows.
\end{proof}

Next, we address the case of one-sided ladders that are not necessarily $t$-connected.
Recall that one-sided ladders are, by definition, path-connected.

\begin{thm}[One-Sided Ladder Theorem] \label{thm:onesidedpathconnected} 
Let $Y$ be a one-sided ladder. The ring $R_t(Y)$ has only trivial semidualizing modules, i.e., $|\s_0(R_t(Y))| \leq 2$
\end{thm}

\begin{proof}
The field $R_1(Y)=\sfk$ has $\s_0(R_1(Y))=\s_0(\sfk)=\{[\sfk]\}$.
Thus, we may assume that $t>1$, and furthermore that $h=0$ and $k>0$. 
If $Y$ contains no $t$-minors, then $R_t(Y)$ is a polynomial ring over $\sfk$, which is Gorenstein, so $\s_0(R_t(Y))=\{[R]\}$ in this case.
Thus, we assume that $Y$ contains a $t$-minor.
Since $Y$ is path-connected, it is straightforward to show that $X_{11}\in Y$ and, moreover, that all the variables $X_{1j}$ and $X_{i1}$ are in $Y$. 

Let $j_1 = \max\{j \mid c_j<t\}$ and
$j_2=\min\{j \mid d_j<t\}$. If $j_1 \geq j_2$, then $Y$ contains no $t$-minors,
so we must have $j_1<j_2$. Let $Y' = \{ X_{ij} \in Y \mid i \leq c_{j_2}
\text{ and } j \leq d_{j_1}\}$ and $Z=Y \ssm Y'$. Then $Y'$ is a one-sided
$t$-connected ladder and $Z$ contains no $t$-minors.
It follows that $R_t(Y)=R_t(Y')[Z]$.  Then $|\frak S_0(R_t(Y))| = |\frak S_0(R_t(Y'))|$ by \cite[Corollary 3.11(a)]{SSW}.  Now apply Corollary \ref{generalonesided}.
\end{proof}

We end this section with an example that illustrates two aspects of Theorem~\ref{thm:onesidedpathconnected} and its proof.

\begin{ex}\label{ex180710a}
The following ladder $Y$ is 2-connected and path-connected.
\begin{equation*}
Y:\quad
\begin{smallmatrix}
\displaystyle X_{11} &\displaystyle X_{12}&\displaystyle X_{13}&\displaystyle X_{14}&\displaystyle X_{15}  \\
\displaystyle X_{21} &\displaystyle X_{22}&\displaystyle X_{23}&\displaystyle X_{24} &\displaystyle X_{25} \\
\displaystyle X_{31}&\displaystyle X_{32}&\displaystyle X_{33} &\displaystyle X_{34}
\end{smallmatrix}
\end{equation*}
However, it is 3-disconnected because the variables $X_{14},X_{24}$ are not used in any 3-minor. 
In the notation of the proof of Theorem~\ref{thm:onesidedpathconnected} with $t=3$, this yields $Z=\{X_{14},X_{24}\}$ and
$Y'$ is the next ladder which is 3-connected 
\begin{equation*}
Y':\quad
\begin{smallmatrix}
\displaystyle X_{11} &\displaystyle X_{12}&\displaystyle X_{13} &\displaystyle X_{14} \\
\displaystyle X_{21} &\displaystyle X_{22}&\displaystyle X_{23} &\displaystyle X_{24}\\
\displaystyle X_{31}&\displaystyle X_{32}&\displaystyle X_{33} &\displaystyle X_{34}
\end{smallmatrix}
\end{equation*}
and $R_3(Y)=R_3(Y')[Z]$, so
$|\s_0(R_3(Y))|=|\s_0(R_3(Y'))|=2$ by~\cite[Theorem 4.2]{SSW}.
Similarly, we have $|\s_0(R_4(O))|=1$ for the ladder $O$ from the introduction.
Also, the path-connected condition in our definition of ``one-sided'' is necessary for Theorem~\ref{thm:onesidedpathconnected}
as the next ladder has no corners but $|R_2(Y'')|=4$ by~\cite[Theorem~4.5]{SSW}.%
\begin{equation*}
Y'':\quad
\begin{smallmatrix}
&&&&\displaystyle X_{14} &\displaystyle X_{15}&\displaystyle X_{16}  \\
&&&&\displaystyle X_{24} &\displaystyle X_{25}&\displaystyle X_{26} \\
\displaystyle X_{31}&\displaystyle X_{32}&\displaystyle X_{33} \\
\displaystyle X_{41}&\displaystyle X_{42}&\displaystyle X_{43} 
\end{smallmatrix}
\end{equation*}
\end{ex}


\section{Size-2 Minors of Two-Sided Ladders with No Coincidental Corners}
\label{sec180114c}


In this section, we study ladders which are 2-connected.  (Minors of size $2 \times 2$ are special in the sense that $R_2(Y)$ is an algebra with straightening laws, or ASL, on the poset $Y$, as per \cite[p.~121]{Co}, but $R_{t>2}(Y)$ is not.  We will consider more general ladders in another project \cite{SWSeSp2}, including ladders with coincidental corners.)
In particular, throughout this section, $Y$ will be a 2-connected ladder without coincidental corners and $R_2(Y)$ the associated ladder determinantal ring.  

As in the previous section, we will use the notation $\ti{Y}$ for ladders obtained from the given ladder $Y$.  The notation $\ti{R}$ will always denote the associated ladder determinantal ring $R_2(\ti{Y})$. See Notation \ref{notn:ladder}.  In order to provide an upper bound on $|\frak S_0(R_2(Y))|$ we will need the additional notation defined below.  

\begin{notn}
  For any ladder, let $\eta_1 = \min\{j \mid b_j\leq d_k\}$, $\eta_2 =\max\{i \mid a_i \leq c_1\}$,
  $\kappa_1 = \min\{i \mid c_i \geq a_h\}$, and $\kappa_2=\max\{j \mid d_j \geq b_1\}$.  For example, in the following ladder, we have $h=7$, $k=8$,
  $\eta_1=3$, $\eta_2=5$, $\kappa_1=4$ and $\kappa_2=6$.

\begin{center}  
  \begin{picture}(204,204)
    \put(0,0){\line(0,1){88}}
    \put(0,88){\line(1,0){16}}
    \put(16,88){\line(0,1){16}}
    \put(16,104){\line(1,0){16}}
    \put(32,104){\line(0,1){20}}
    \put(32,124){\line(1,0){16}}
    \put(48,124){\line(0,1){16}}
    \put(48,140){\line(1,0){16}}
    \put(64,140){\line(0,1){16}}
    \put(64,156){\line(1,0){16}}
    \put(80,156){\line(0,1){16}}
    \put(80,172){\line(1,0){16}}
    \put(96,172){\line(0,1){16}}
    \put(96,188){\line(1,0){16}}
    \put(112,188){\line(0,1){16}}
    \put(112,204){\line(1,0){92}}
    \put(0,0){\line(1,0){88}}
    \put(88,0){\line(0,1){16}}
    \put(88,16){\line(1,0){16}}
    \put(104,16){\line(0,1){16}}
    \put(104,32){\line(1,0){16}}
    \put(120,32){\line(0,1){16}}
    \put(120,48){\line(1,0){16}}
    \put(136,48){\line(0,1){16}}
    \put(136,64){\line(1,0){16}}
    \put(152,64){\line(0,1){16}}
    \put(152,80){\line(1,0){20}}
    \put(172,80){\line(0,1){16}}
    \put(172,96){\line(1,0){8}}
    \put(180,96){\line(0,1){8}}
    \put(180,104){\line(1,0){8}}
    \put(188,104){\line(0,1){10}}
    \put(188,114){\line(1,0){16}}
    \put(204,114){\line(0,1){90}}
    \multiput(112,32)(0,8){20}{\line(0,1){4}}
    \multiput(88,16)(0,8){20}{\line(0,1){4}}
    \multiput(32,114)(8,0){20}{\line(1,0){4}}
    \multiput(16,88)(8,0){20}{\line(1,0){4}}
    \put(79,193){$(a_1,b_1)$}
    \put(47,161){$(a_3,b_3)$}
    \put(15,129){$(a_5,b_5)$}
    \put(-17,93){$(a_7,b_7)$}
    \put(190,103){$(c_1,d_1)$}
    \put(155,69){$(c_4,d_4)$}
    \put(123,37){$(c_6,d_6)$}
    \put(91,5){$(c_8,d_8)$}
  \end{picture}
\end{center}

As a second example, for the ladder \ref{eq180304a} on page~\pageref{eq180304a}, $\eta_1 = 1; \eta_2 = 0; \kappa_1 = 2$; and $\kappa_2 = 2$.

\end{notn}

\begin{rmk} Note that a ladder is one-sided if and only if $\eta_1 = 0$ or $\kappa_1 = 0$; i.e., if and only if $k=0$ or $h=0$, respectively.
\end{rmk}

\begin{prop} \label{prop:4corners}
  Let $R = R_2(Y)$ for a two-sided 2-connected ladder $Y$ with $h \geq 1$ lower inside corners and $k \geq 1$ upper inside corners, such that no two inside corners
  coincide.  Assume that for all 2-connected ladders $Z$ with fewer than $h+k$
  inside corners, where no two coincide, the associated ladder determinantal ring $R_2(Z)$ has only trivial
  semidualizing modules. Then $|\frak S_0(R)| \leq 4$.
\end{prop}

\begin{proof}
  As per Fact \ref{omega}, $[\omega_R] = \sum_{i=1}^{h+1} \lambda_i [\q_i] + \sum_{j=1}^k \delta_j[\p_j]$.  The letters $M_i,N_i$ will be used to denote (possible) semidualizing modules of $R$.  First, we invert $x_{a_0b_1}$ and
  obtain the ladder $\ti{Y}$ by deleting rows $a_0,a_0+1\dots,a_1-1$ and columns
  $b_1+1,b_1+2\dots,b_0$ of $Y$. The kernel of the natural map $\varphi \colon
  \Cl(R) \to \Cl(\ti{R})$ is generated by $[\q_1],[\p_1],[\p_2],\dots,
  [\p_{\kappa_2}]$. (Note that it is possible for $\kappa_2 = 0$, in which case, the kernel is generated only by $[\q_1]$.  The argument below allows for this possibility.) By assumption, the semidualizing modules of $\ti{R}$ are
  $[\ti{R}]$ and 
  its canonical class, hence the possible semidualizing modules of
  $R$ are
  \begin{align*}
    [N_1] &= r_1[\q_1] + \sum_{j=1}^{\kappa_2}s_j[\p_j], \text{ and}\\
    [N_2] &= r_1[\q_1] + \sum_{j=1}^{\kappa_2}s_j[\p_j]
    + \sum_{i=2}^{h+1}\gl_i[\q_i] + \sum_{j=\kappa_2+1}^k \gd_j[\p_j],
  \end{align*}
  where $r_1,s_j \in \bbz$ and $[\omega_{\ti{R}}]=\sum_{i=1}^{h}\gl_{i+1}[\ti{\q}_i] + \sum_{j=1}^{k-\kappa_2} \gd_{j+\kappa_2}[\ti{\p}_j]$.
  
  Next, we invert $x_{a_hb_{h+1}}$ and obtain (a new) $\ti{Y}$ by deleting rows
  $a_h+1,a_h+2,\dots,a_{h+1}$ and columns $b_{h+1},b_{h+1}+1,\dots,b_h-1$
  of $Y$. The kernel of the natural map $\varphi \colon \Cl(R) \to \Cl(\ti{R})$
  is generated by $[\q_{h+1}],[\p_{\kappa_1}],[\p_{\kappa_1+1}],\dots,[\p_k]$, and $[\omega_{\ti R}] = \sum_{i=1}^h\gl_i[\ti{\q}_i] + \sum_{j=1}^{\kappa_1-1} \gd_j[\ti{\p}_j]$.
  
  Suppose that $\varphi([N_1])=0$. If $\kappa_1 > \kappa_2$, then $0=\varphi([N_1])= r_1[\ti{\q}_1] + \sum_{j=1}^{\kappa_2}s_j[\ti{\p}_j]$, so $[N_1]=0$. Since we are seeking nontrivial semidualizing modules, we may assume here that $\kappa_1\leq \kappa_2$, in which case  $0=\varphi([N_1])= r_1[\ti{\q}_1] + \sum_{j=1}^{\kappa_1-1} s_j[\ti{\p}_j], $ and hence, $[N_1]$ equals
  \[
    [N_3] = \sum_{j=\kappa_1}^{\kappa_2} s_j[\p_j], \text{ where } \kappa_1 \leq \kappa_2.
  \]
  
  Suppose that $\varphi([N_1])=[\omega_{\ti{R}}]$.  Because neither relation among the $\kappa_i$ may be discarded, we allow for both cases (where the notation $\sum_r^s$ for $s < r$ is simply a vacuous sum).  Then $[N_1]$ equals
   \[
    [N_4] = \gl_1[\q_1]+\sum_{j=1}^{\min(\kappa_1-1,\kappa_2)}\gd_j[\p_j]
    +\sum_{j=\kappa_1}^{\kappa_2} s_j[\p_j],
  \]
  where $\gl_i=0$ for all $1<i<h+1$, and $\gd_j=0$ for $\kappa_2<j$ and $j<\kappa_1$, a condition which may or may not be satisfied.  (In particular, it's not satisfied if $\kappa_1 \leq \kappa_2 + 1$.)

  Suppose that $\varphi([N_2])=0$. Because neither relation among the $\kappa_i$ may be discarded, we allow for both cases.  We have $\varphi([N_2]) = r_1[\ti{\q}_1] + \sum_{j=1}^{\min(\kappa_1-1, \kappa_2)}s_j[\ti{\p}_j] +  \sum_{i=2}^h\gl_i[\ti{\q}_i] + \sum_{j=\kappa_2+1}^{\max(\kappa_1-1, \kappa_2)} \gd_j[\ti{\p}_j]$, hence, $[N_2]$ equals
  \[
    [N_5] = \gl_{h+1}[\q_{h+1}] + \sum_{j=\kappa_1}^{\kappa_2} s_j[\p_j]
    + \sum_{j=\max(\kappa_1-1,\kappa_2)+1}^k \gd_j[\p_j],
  \]
  where $\gl_i=0$ for all $1<i<h+1$, and $\gd_j=0$ for $\kappa_2<j$ and $j<\kappa_1$, a condition which may or may not be satisfied.  (In particular, it's not satisfied if $\kappa_1 \leq \kappa_2 + 1$.)

  Suppose that $\varphi([N_2])=[\omega_{\ti{R}}]$. If $\kappa_1 > \kappa_2$, then
  $[N_2]=[\omega_R]$. So we may assume that $\kappa_1\leq \kappa_2$, in which case $[N_2]$
  equals
  \[
    [N_6] = \sum_{j=1}^{h+1}\gl_i[\q_i]+\sum_{j=1}^{\kappa_1-1}\gd_j[\p_j]
    +\sum_{j=\kappa_1}^{\kappa_2} s_j[\p_j]+\sum_{j=\kappa_2+1}^k \gd_j[\p_j],
    \text{ where } \kappa_1\leq \kappa_2.
  \]
  
  Now we invert $x_{c_1d_0}$ and obtain $\ti{Y}$ by deleting rows
  $c_0,c_0+1,\dots,c_1-1$ and columns $d_1+1,d_1+2,\dots,d_0$ of $Y$. The kernel
  of the natural map $\varphi \colon \Cl(R) \to \Cl(\ti{R})$
  is generated by $[\q_1],[\q_2],\dots,[\q_{\eta_2+1}]$, and $\varphi([\p_1])
  =[\ti{\q}_1]$. Let us write $[\ti{\p}_0]=[\ti{\q}_1]$.
  We have
  \[
  [\omega_{\ti{R}}] = \sum_{i=2}^{h-\eta_2+1}\gl_{i+\eta_2}[\ti{\q}_i] + \delta_1 [\ti{\q}_1] + \sum_{j=1}^{k-1} \gd_{j+1}[\ti{\p}_j].
  \]

   If $\varphi([N_3])=0$, then $0=\varphi([N_3])=\sum_{j=\kappa_1}^{\kappa_2} s_j[\ti{\p}_{j-1}]$ implies that $[N_3]=0$.  If $\varphi([N_3])=
  [\omega_{\ti{R}}]$, then $[N_3]$ gives us the possibly nontrivial
  semidualizing module
  \[
    [N_7] = \sum_{j=\kappa_1}^{\kappa_2} \gd_j[\p_j], \text{ where } \kappa_1 \leq \kappa_2,
  \]
  and $\gl_i=0$ for all $i>\eta_2+1$, and $\gd_j=0$ for all $j<\kappa_1$ or $j>\kappa_2$.
  
  
   If $\varphi([N_4])=0$, then since $0=\varphi([N_4]) = \sum_{j=1}^{\min(\kappa_1-1,\kappa_2)}\gd_j[\ti{\p}_{j-1}]
    +\sum_{j=\kappa_1}^{\kappa_2} s_j[\ti{\p}_{j-1}]$, $[N_4]$ gives us the candidate
  \[
    [N_8] = \gl_1[\q_1],
  \]
   where $\gl_i=0$ for all $1<i<h+1$, and $\gd_j=0$ for all $j$ such that $\kappa_2<j<\kappa_1$ or $1\leq j \leq
  \min(\kappa_1-1,\kappa_2)$. Whether or not $\kappa_1 > \kappa_2$ or $\kappa_1 \leq \kappa_2$, we conclude that $\gd_j = 0$ for all $j < \kappa_1$.
  
Suppose that $\varphi([N_4])=[\omega_{\ti{R}}]$. If $\eta_2<h$, then $\lambda_i = 0$ for all $i >1$, and $[N_4]=
  [\omega_R]$.  Therefore, we may assume that $\eta_2=h$.  Thus, the
  case $[N_4]$ gives us the candidate
  \[
    [N_9] = \gl_1[\q_1] + \sum_{j=1}^{\kappa_2} \gd_j[\p_j],
  \]
  where $\eta_2=h$, which implies $\kappa_1 = 1$, and $\gl_i=0$ for all $1<i<h+1$ and $\gd_j=0$ for $j>\kappa_2$.


   Suppose that $\varphi([N_5])=0$.  If $\eta_2<h$, then none of the terms in the expression for $[N_5]$ is in $\ker \varphi$, in which case $[N_5] = 0$. Therefore, we may assume that $\eta_2=h$, in which case the only term in the expression for $[N_5]$ that is in $\ker \varphi$ is $\gl_{h+1}[\q_{h+1}]$.  Thus, $[N_5]$ gives us the candidate
     \[
    [N_{10}] = \gl_{h+1}[\q_{h+1}] = [\omega_R]-[N_9].
  \]
   
   Suppose that $\varphi([N_5])=[\omega_{\ti R}]$.  If $\eta_2 < h$ and $\kappa_1 \leq \kappa_2$,
   then we have
  \[
  \varphi([N_5]) = \gl_{h+1}[\ti{\q}_{h-\eta_2+1}] + \sum_{j=\kappa_1-1}^{\kappa_2-1}s_{j+1}[\ti{\p}_j] + \sum_{j=\kappa_2}^{k-1} \gd_{j+1}[\ti{\p}_j].
  \]
  If $\eta_2 < h$ and $\kappa_1 > \kappa_2$, then we have
  \[
  \varphi([N_5]) = \gl_{h+1}[\ti{\q}_{h-\eta_2+1}] + \sum_{j=\kappa_1-1}^{k-1} \gd_{j+1}[\ti{\p}_j].
  \]
  If $\eta_2 = h$ and $\kappa_1 \leq \kappa_2$, then we have
  \[
     [\omega_{\ti{R}}] = \delta_1 [\ti{\q}_1] + \sum_{j=1}^{k-1} \gd_{j+1}[\ti{\p}_j] \quad \text{and} \quad
     \varphi([N_5]) = \sum_{j=\kappa_1-1}^{\kappa_2-1}s_{j+1}[\ti{\p}_j] + \sum_{j=\kappa_2}^{k-1} \gd_{j+1}[\ti{\p}_j].
  \]
  If $\eta_2 = h$ and $\kappa_1 > \kappa_2$, then we have
  \[
     [\omega_{\ti{R}}] = \delta_1 [\ti{\q}_1] + \sum_{j=1}^{k-1} \gd_{j+1}[\ti{\p}_j] \quad \text{and} \quad
     \varphi([N_5]) = \sum_{j=\kappa_1-1}^{k-1} \gd_{j+1}[\ti{\p}_j].
  \]
  In all cases, $[N_5]$ gives us the candidate
       \[
    [N_{11}] = \gl_{h+1}[\q_{h+1}] + \sum_{j=\kappa_1}^k \gd_j[\p_j]  = [\omega_R]-[N_8],  
    \]
  where $\gl_i=0$ for all $1<i<h+1$ and $\gd_j = 0$ for $j < \kappa_1$.

Suppose that $\varphi([N_6])=0$.  Then
\[
  0=\varphi([N_6]) = \sum_{i=2}^{h-\eta_2+1}\gl_{i+\eta_2}[\ti{\q}_i] + \sum_{j=1}^{\kappa_1-1}\gd_j[\ti{\p}_{j-1}]
    +\sum_{j=\kappa_1}^{\kappa_2} s_j[\ti{\p}_{j-1}]+\sum_{j=\kappa_2+1}^k \gd_j[\ti{\p}_{j-1}],
\]
  where $\kappa_1 \leq \kappa_2$. So $[N_6]$ gives us the candidate     
  \[
    [N_{12}] = \sum_{i=1}^{\eta_2+1} \gl_i[\q_i] = [\omega_R]-[N_7], \text{ where } \kappa_1 \leq \kappa_2,
  \]  
$\gl_i=0$ for all $i> \eta_2+1$ and $\gd_j=0$ for all $j < \kappa_1$ or $j > \kappa_2$.  

 If $\varphi([N_6])=[\omega_{\ti{R}}]$, then $[N_6] = [\omega_R]$.

  Finally, we invert $x_{c_{k+1}d_k}$ and obtain $\ti{Y}$ by deleting rows
  $c_k+1,c_k+2,\dots,c_{k+1}$ and columns $d_{k+1},d_{k+1}+1,\dots,d_k-1$
  of $Y$. Then $\Cl(\ti{R})$ is generated by the basis elements 
  $[\ti{\q}_1],[\ti{\q}_2],\dots,[\ti{\q}_{\eta_1}],[\ti{\p}_1],[\ti{\p}_2],
  \dots,[\ti{\p}_{k-1}]$, and
  \begin{alignat*}{2}
    [\omega_{\ti{R}}] &= \sum_{i=1}^{\eta_1} \ti{\gl}_i[\ti{\q}_i]
    + \sum_{j=1}^{k-1} \ti{\gd}_j[\ti{\p}_j], &&\text{ where}\\
    \ti{\gl}_i &= \gl_i &&\text{ for all } i<\eta_1,\\
    \ti{\gl}_{\eta_1} &= c_k+d_k-a_{\eta_1-1}-b_{\eta_1-1},\\
    \ti{\gd}_j &= \gd_j &&\text{ if } c_j<a_{\eta_1-1}, \text{ and}\\
    \ti{\gd}_j &= c_k+d_k-c_j-d_j &&\text{ otherwise.}
  \end{alignat*}
  
  Suppose that $\varphi([N_7])=0$ (equivalently, $\varphi([N_{12}])=
  [\omega_{\ti{R}}]$) under the natural map $\varphi \colon \Cl(R)
  \to \Cl(\ti{R})$. If $\kappa_2\neq k$, then $0=\varphi([N_7]) = \sum_{j= \kappa_1}^{\kappa_2} \gd_j[\ti{\p}_j]$, so $[N_7]=0$. If $\kappa_2=k$
  (equivalently, $\eta_1= 1$), then 
  \[
    0=\varphi([N_7])=\gd_k[\ti{\q}'_1]+ \sum_{j=\kappa_1}^{k-1} \gd_j[\ti{\p}_j]
    = -\gd_k[\ti{\q}_1]+\sum_{j=\kappa_1}^{k-1} (\gd_j-\gd_k)[\ti{\p}_j],
  \]
  where the last equality follows from \cite[Corollary 2.3(i), with $I_1 = \{1, \dots, k-1\}$]{Co}.  Since 
  the $[\ti{\q}_1], [\ti{\p}_j]$ are basis elements, we again get $[N_7]=0$.  Thus, $\varphi([N_7])=0$ produces no candidate for a semidualizing module.
  
  Suppose that $\varphi([N_{12}])=0$ (equivalently, $\varphi([N_7])=
  [\omega_{\ti{R}}]$). If $\eta_2 < \eta_1$, then  $\varphi([N_{12}]) =\sum_{i=1}^{\eta_2+1} \gl_i[\ti{\q}_i]$ implies that 
    $[N_{12}]=0$.
  If $\eta_2\geq \eta_1$, then
  \[
    0 = \varphi([N_{12}]) = \sum_{i=1}^{\eta_1-1} \gl_i[\ti{\q}_i]
    + \left( \sum_{i=\eta_1}^{\eta_2+1} \gl_i \right) [\ti{\q}_{\eta_1}],
  \]
  (where if $\eta_1 = 1$, then the first sum is vacuous).  Since 
  the $[\ti{\q}_1], \dots, [\ti{\q}_{\eta_1}]$ are basis elements, we must have $\lambda_i = 0$ for $1 \leq i \leq \eta_1- 1$ and $\sum_{i=\eta_1}^{\eta_2+1} \lambda_i= 0$.
  
Thus, $[N_{12}]$, and hence its pair $[N_7]$, give us, respectively, the candidates
  \[
    [M_1] = \sum_{i=\eta_1}^{\eta_2+1} \gl_i[\q_i] \quad \text{and} \quad
    [M_2] = \sum_{j=\kappa_1}^{\kappa_2} \gd_j[\p_j] = [\omega_R]-[M_1],
  \]
  where $\eta_1\leq \eta_2$, $\kappa_1\leq \kappa_2$, $\gl_i=0$ for all $i<\eta_1$ or $i>\eta_2+1$,
  $\gd_j=0$ for all $j<\kappa_1$ or $j>\kappa_2$, and $\gl_{\eta_1}+\gl_{\eta_1+1}+\dots+
  \gl_{\eta_2+1}=0$. In other words, the corners
  $(a_0,b_0)$, $(a_{h+1},b_{h+1})$ together with all inside corners,
  except $(a_i,b_i),(c_j,d_j)$ for $\eta_1 \leq i \leq \eta_2$ and
  $\kappa_1 \leq j \leq \kappa_2$, all lie on the same antidiagonal. Furthermore,
  we have $a_{\eta_1} \leq a_{\eta_2} \leq c_1 \leq c_{\kappa_1} \leq c_{\kappa_2}$ and
  $b_{\eta_2} \leq b_{\eta_1} \leq d_k \leq d_{\kappa_2} \leq d_{\kappa_1}$, and by
  assumption, no two inside corners coincide. Hence
  $(a_i,b_i) \lneqq (c_j,d_j)$ for all $\eta_1 \leq i \leq \eta_2$ and
  $\kappa_1 \leq j \leq \kappa_2$.
  
  
  Suppose that $\varphi([N_{10}])=0$ (equivalently, $\varphi([N_9])=
  [\omega_{\ti{R}}]$). Since $\eta_2=h$ in this case, we have
  $0=\varphi([N_{10}])=\gl_{h+1}[\ti{\q}_{\eta_1}]$, so $[N_{10}]=0$. Thus, $\varphi([N_{10}])=0$ produces no candidate for a semidualizing module.

  Suppose that $\varphi([N_9])=0$ (equivalently, $\varphi([N_{10}])=
  [\omega_{\ti{R}}]$). If $\kappa_2<k$, then  $\varphi([N_9])=\gl_1[\ti{\q}_1]+ \sum_{j=1}^{\kappa_2} \gd_j[\ti{\p}_j]$ implies that $[N_9]=0$ since the $[\ti{\q}_i], [\ti{\p}_j]$ are basis elements. If $\kappa_2=k$, then
  \[
    0 = \varphi([N_9]) = \gl_1[\ti{\q}_1] + \gd_k[\ti{\q}'_1]
    +\sum_{j=1}^{k-1} \gd_j[\ti{\p}_j]
    = (\gl_1-\gd_k)[\ti{\q}_1]+\sum_{j=1}^{k-1} (\gd_j-\gd_k)[\ti{\p}_j],
  \]
   where the last equality follows from \cite[Corollary 2.3(i), with $I_1 = \{1, \dots, k-1\}$]{Co}.  Since 
  the $[\ti{\q}_1], [\ti{\p}_j]$ are basis elements, it follows that $\gl_1=\gd_1=\gd_2=\dots=\gd_k$.  Therefore, $[N_9],[N_{10}]$ give us the
  candidates
  \[
    [M_3] = \gl_1[\q_1] + \sum_{j=1}^k \gl_1[\p_j]
    =-\gl_1[\q'_1] \quad \text{and} \quad
    [M_4] = \gl_{h+1}[\q_{h+1}] = [\omega_R]-[M_3],
  \]
  where $\eta_2=h$, $\kappa_1=1$, $\kappa_2=k$, hence $\eta_1=1$, and $\gl_1=\gd_1=\gd_2=\dots=\gd_k$,
  and $\gl_2=\gl_3=\dots=\gl_h=0$. Since $\eta_1=\kappa_1=1$, we have $(a_i,b_i)
  \lneqq (c_j,d_j)$ for all $1 \leq i \leq h$ and $1 \leq j \leq k$.
  
  
  If $\varphi([N_8])=0$, then $\varphi([N_8])= \gl_1[\ti{\q}_1]$ implies $[N_8]=0$. Suppose that $\varphi([N_8])
  =[\omega_{\ti{R}}]$. In this case, we have $\gl_i=\gd_j=0$ for all $1<i<h+1$
  and $j<\kappa_1$, $\gl_1=\ti{\gl}_1$, and $\ti{\gl}_i=\ti{\gd}_j=0$ for all
  $1<i<\eta_1+1$ and $j<k$. Since $\gl_i=0$ for all $1<i<h+1$, the corners
  $(a_i,b_i)$ for $1\leq i\leq h$ all lie on the same antidiagonal. Since $\gl_1=\ti{\gl}_1$ and $\ti{\gl}_i=0$ for all $1<i<\eta_1+1$,
  where $\ti{\gl}_i=\gl_i=a_i+b_i-a_{i-1}-b_{i-1}$ for
  $2 \leq i \leq \eta_1-1$ and
  $\ti{\gl}_{\eta_1} = c_k+d_k-a_{\eta_1-1}-b_{\eta_1-1}$, the corners
  $(a_1,b_1),\ldots,(a_h,b_h)$ and $(c_k,d_k)$ lie on the same antidiagonal (whether or not $\eta_1 = 1$). By definition (recall Fact \ref{omega}), $\gd_j:=
  a_{i_j}+b_{i_j}-c_j-d_j= 0$ for
  all $j<\kappa_1 = \min\{i \mid c_i \geq a_h\}$
  i.e., $i_j \leq h$, where $i_j=\min\{i\colon a_i>c_j\}$, the corners $(a_1,b_1),\ldots, (a_h,b_h)$, $(c_k,d_k)$
  and $(c_j,d_j)$ for $1\leq j<\kappa_1$ all
  lie on the same antidiagonal. Since $\ti{\gd}_j=c_k+d_k-c_j-d_j=0$ for all
  $j$ such that $c_j\geq a_{\eta_1-1}$, and $c_{\kappa_1}\geq a_h \geq a_{\eta_1-1}$, the
  corners $(c_j,d_j)$ for $\kappa_1 \leq j \leq k$ all lie on the same antidiagonal. Thus, $\lambda_{h+1} := a_{h+1} + b_{h+1} -a_h - b_h = a_{h+1} + b_{h+1} - c_j - d_j= \delta_j$ for $\kappa_1 \leq j \leq k$.
  Hence $[N_8]$ and $[N_{11}]$ give us the candidates
      \begin{align*}
    [M_5]&=\gl_1[\q_1] \quad \text{and}\\
    [M_6]&=\gl_{h+1}[\q_{h+1}]+\sum_{j=\kappa_1}^k \gl_{h+1}[\p_j] 
    = -\gl_{h+1}[\q'_{h+1}]
    = [\omega_R]-[M_5],
  \end{align*}
  (where the last equality again follows from \cite[Corollary 2.3(i), with $I_1 = \{1, \dots, k-1\}$]{Co} and) where all inside corners lie on the same antidiagonal.
  
  To summarize, the possible semidualizing modules of $R$ are listed below, along with the conditions in which they have the potential to exist based upon the analysis above:
  \begin{alignat*}{3}
    &[M_1],[M_2],[M_3],[M_4] &\quad & \text{if } (a_i,b_i) \lneqq (c_j,d_j)
    \text{ for all } 1 \leq i \leq h \text{ and } 1 \leq j \leq k,\\
    &[M_5],[M_6] && \text{if } (a_i,b_i) \nleq (c_j,d_j)
    \text{ for all } 1 \leq i \leq h \text{ and } 1 \leq j \leq k, \text{ and}
    \hspace{-9pt}\\
    &[M_1],[M_2] && \text{otherwise.} & \qedhere
  \end{alignat*}
\end{proof}

With this summary in hand, 
it is convenient to address two cases based upon the shape of the ladder.  In particular, we use the descriptives  ``thick'' and ``thin''. The most basic case of each such ladder is outlined below, where there is exactly one lower, and one upper, inside corner.  Casually speaking, a ``thick'' ladder is one in which every lower inside corner $(a_i,b_i)$ is strictly less than every upper inside corner $(c_j,d_j)$ (i.e., the case of $[M_1]-[M_4]$ above), while a  ``thin'' ladder is the diametric opposite of this; i.e., one such that $(a_i,b_i)\nleq (c_j,d_j)$, for all $1 \leq i \leq h$ and $1 \leq j \leq k$.  Note that for the latter, it is possible for upper and lower inside corners to lie on the same antidiagonal.
For the ladder on the right, it is necessary that $a_1<c_1$ if $b_1 > d_1$.  The result for the case $a_1>c_1$ follows by symmetry.
  
  
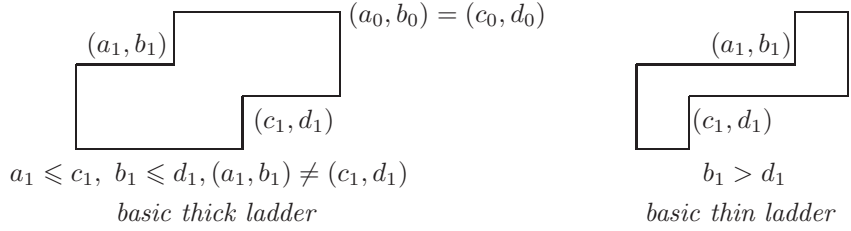
\begin{center}
    \begin{picture}(252,67)
      \put(-20,12){\line(1,0){63}}
      \put(43,12){\line(0,1){20}}
      \put(47,20){$(c_1,d_1)$}
      \put(43,32){\line(1,0){37}}
      \put(80,32){\line(0,1){32}}
      \put(17,64){\line(1,0){63}}
      \put(83,60){$(a_0,b_0)=(c_0,d_0)$}
      \put(17,44){\line(0,1){20}}
      \put(-20,44){\line(1,0){37}}
      \put(-16,49){$(a_1,b_1)$}
      \put(-20,12){\line(0,1){32}}
      \put(-45,0){$a_1 \leq c_1,\ b_1 \leq d_1,(a_1,b_1) \neq (c_1,d_1)$}
      \put(-5,-15){{\it basic thick ladder}}
      \put(192,12){\line(1,0){20}}
      \put(212,12){\line(0,1){20}}
      \put(213,20){$(c_1,d_1)$}
      \put(212,32){\line(1,0){60}}
      \put(272,32){\line(0,1){32}}
      \put(252,64){\line(1,0){20}}
      \put(252,44){\line(0,1){20}}
      \put(192,44){\line(1,0){60}}
      \put(220,49){$(a_1,b_1)$}
      \put(192,12){\line(0,1){32}}
      \put(217,0){$b_1>d_1$}
      \put(195,-15){{\it basic thin ladder}}
    \end{picture}
  \end{center}
\vskip.25in

\begin{defn}
  Let $Y$ be a ladder. By abuse of language, we say that we \textbf{reflect along the
  antidiagonal} if we form the ladder $\ti{Y}$, where
  $\ti{Y}_{ij}= X_{a_{h+1}-j+a_0,b_0-i+b_{h+1}}$.
  The ladder $\ti{Y}$ has corners $(\ti{a}_0,\ti{b}_0) = (b_{h+1},a_{h+1})$,
  $(\ti{a}_1,\ti{b}_1) = (b_0-d_1+b_{h+1},a_{h+1}-c_1+a_0)$, \ldots,
  $(\ti{a}_k,\ti{b}_k) = (b_0-d_k+b_{h+1},a_{h+1}-c_k+a_0)$,
  $(\ti{a}_{k+1},\ti{b}_{k+1}) = (b_0,a_0)$,
  $(\ti{c}_1,\ti{d}_1) = (b_0-b_1+b_{h+1},a_{h+1}-a_1+a_0)$, \ldots,
  $(\ti{c}_h,\ti{d}_h) = (b_0-b_h+b_{h+1},a_{h+1}-a_h+a_0)$.
\end{defn}

\begin{thm}[Thick Ladder Theorem] \label{thm:thick} 
  Let $Y$ be a two-sided 2-connected ladder, with $h\geq 1$ lower inside corners and $k\geq 1$ upper inside corners, such that $(a_i,b_i) \lneqq (c_j,d_j)$
  for all $1 \leq i \leq h$ and $1 \leq j \leq k$.
  Let $R = R_2(Y)$. Then $|\s_0(R)| \leq 2$.
\end{thm}

\begin{proof}
  By Theorem~\ref{thm:onesided} and the proof of Proposition~\ref{prop:4corners}, we
  only need to show that $[M_1],[M_2],[M_3],[M_4]$ in Proposition~\ref{prop:4corners}
  must be trivial semidualizing modules.
  
  \setcounter{case}{0}
  \begin{case}
  Let us first consider
  \[
    [M_3] = \gl_1[\q_1] + \sum_{j=1}^k \gl_1[\p_j]
    =-\gl_1[\q'_1] \quad \text{and} \quad
    [M_4] = \gl_{h+1}[\q_{h+1}] = [\omega_R]-[M_3],
  \]
  where $\gl_1=\gd_1=\gd_2=\dots=\gd_k$ and $\gl_2=\gl_3=\dots=\gl_h=0$.
  In this case, $\gl_1 = a_{h+1}+b_{h+1}-c_k-d_k < a_{h+1}+b_{h+1}-a_1-b_1 =
  \gl_{h+1}$.

  \begin{subcase} \label{case:gamma1positive}
    Suppose that $\gl_1>0$. Write $M_3 = \q_1^{\gl_1} \cap \p_1^{\gl_1} \cap
    \dots \cap \p_k^{\gl_1}$ and $M_4 = \q_{h+1}^{\gl_{h+1}}$.
    Under the multiplication map of ideals $\mu \colon M_3 \otimes_R M_4 \to
    M_3 M_4$, we have 
    \begin{align*}
      \mu( x_{a_0b_1}^{\gl_1-1} x_{a_0b_0} x_{a_h b_{h+1}}\otimes
      x_{a_hb_1}^{\gl_{h+1}}) 
       = x_{a_hb_1}^{\gl_{h+1}-1} x_{a_0 b_1}^{\gl_1-1} x_{a_h b_1}
      x_{a_0 b_0} x_{a_h b_{h+1}}\\
      = x_{a_hb_1}^{\gl_{h+1}-1}x_{a_0b_1}^{\gl_1}x_{a_hb_0}x_{a_h b_{h+1}}
       {\text{ since }}  x_{a_h b_1}x_{a_0 b_0} = x_{a_h b_0}x_{a_0 b_1} \\
      = \mu( x_{a_0b_1}^{\gl_1} \otimes
      x_{a_hb_1}^{\gl_{h+1}-1} x_{a_hb_0} x_{a_h b_{h+1}}).
    \end{align*}
    Hence $\mu$ is not injective, contradicting Fact~\ref{fact:multisiso}.
  \end{subcase}

  \begin{subcase}
    Suppose that $\gl_{h+1}>0$ and $\gl_1<0$. Let $M_3 = {\q'_1}^{-\gl_1}$,
    $M_4=\q_{h+1}^{\gl_{h+1}}$, and $\omega_R = {\q'_1}^{-\gl_1} \cap
    \q_{h+1}^{\gl_{h+1}}$. As in Case \ref{case:mingen} of Theorem \ref{thm:onesided}, we get a contradiction since the function
    \newcommand{\lcm}{\operatorname{lcm}}%
\[
  \lcm \colon \mingen((\q'_1)^{|\lambda_1|}) \times \mingen(\q_{h+1}^{\lambda_{h+1}}) \to (\q'_1)^{|\lambda_1|} \cap \q_{h+1}^{\lambda_{h+1}}
\]
does not give a bijection of minimal generating sets.
  \end{subcase}
  
  \begin{subcase}
    Suppose that $\gl_{h+1}<0$. We reflect along the antidiagonal to get
    \begin{align*}
      [\omega_{\ti{R}}] &= (a_0+b_0-c_1-d_1)[\ti{\q}_1] +
      (c_k+d_k-a_{h+1}-b_{h+1})[\ti{\q}_{k+1}]\\
      &\hspace{12pt} {} + (a_1+b_1-a_{h+1}-b_{h+1}) [\ti{\p}_1] + \dots +
      (a_h+b_h-a_{h+1}-b_{h+1}) [\ti{\p}_h]\\
      &= -\gl_1 [\ti{\q}_{k+1}] - \gl_{h+1} \left( [\ti{\q}_1]
      + \sum_{j=1}^h [\ti{\p}_j] \right)
    \end{align*}
    We may then use Case~\ref{case:gamma1positive} to reach our contradiction.
  \end{subcase}
  
  \end{case}
  
  \begin{case}
    Now we consider the candidates $[M_1], [M_2]$, where, as a result of the hypotheses, $\eta_1 = 1; \eta_2 = h; \kappa_1 = 1$, and $\kappa_2 = k$.  Thus, $[M_1] = \sum_{i=1}^{h+1} \gl_i [\q_i]$ and
    $[M_2] = [\omega_R] - [M_1] = \sum_{j=1}^k \gd_j [\p_j]$, where
    $\sum_{i=1}^{h+1} \gl_i =0$, i.e.\ $a_0+b_0=a_{h+1}+b_{h+1}$. We note that
    $a_i+b_i < c_j+d_j$ for all $1 \leq i \leq h$ and $1 \leq j \leq k$.
  
  \begin{subcase} \label{case:deltanonneg}
    Suppose that $c_j+d_j \leq a_{h+1}+b_{h+1}$ for all $1 \leq j \leq k$.
    Then $\delta_j \geq 0$ for all $1 \leq j \leq k$ and $\gl_{h+1} > 0$.
    Let us write
    \begin{align*}
      [M_1] &= \sum_{\gl_i > 0} \gl_i [\q_i] + \sum_{\gl_i < 0} -|\gl_i|
      [\q_i]\\
      &= \sum_{\gl_i > 0} \gl_i [\q_i] + \sum_{\gl_i < 0} |\gl_i|
      \left( [\q'_i] + \sum_{j=1}^k [\p_j] \right)\\
      &= \sum_{\gl_i > 0} \gl_i [\q_i] + \sum_{\gl_i < 0} |\gl_i| [\q'_i]
      + \sum_{j=1}^k \sum_{\gl_i < 0} |\gl_i| [\p_j]
    \end{align*}
    We let
    \begin{align*}
      M_1 &= \bigcap_{\gl_i > 0} \q_i^{\gl_i} \cap \bigcap_{\gl_i < 0}
      \left( \q'_i \right)^{|\gl_i|} \cap \bigcap_{j=1}^k
      \p_j^{\sum_{\gl_i < 0} |\gl_i|} \\
      M_2 &= \bigcap_{j=1}^k \p_j^{\delta_j}
    \end{align*}
    Let $r = \max \{ \delta_j \mid 1 \leq j \leq k \}$. Under the multiplication
    map of ideals $\mu \colon M_1 \otimes_R M_2 \to M_1 M_2$, we have 
    
     \[
      \mu \left( x_{a_hb_{h+1}}^{\gl_{h+1}} \prod_{i=1}^h x_{a_{i-1}b_i}^{|\gl_i|}
      \otimes x_{a_hb_1}^r \right)
    = \mu \left( x_{a_hb_{h+1}}^{\gl_{h+1}-1} x_{a_hb_1} \prod_{i=1}^h
      x_{a_{i-1}b_i}^{|\gl_i|} \otimes x_{a_hb_1}^{r-1} x_{a_hb_{h+1}} \right).
    \]
    Hence $\mu$ is not injective, contradicting Fact~\ref{fact:multisiso}.
  \end{subcase}
  
  \begin{subcase} \label{case:somedeltaneg}
    Suppose that $c_j+d_j > a_{h+1}+b_{h+1}$ for some $1 \leq j \leq k$
    and $a_i+b_i < a_0+b_0 = a_{h+1}+b_{h+1}$ for all $1 \leq i \leq h$.
    Then $\gl_1 < 0$ and $\gl_h > 0$.
    Let $\delta_{j_0} = \min\{\delta_j \mid 1 \leq j \leq k\}$. Let us write
    \begin{align*}
      [M_2] &= -\left| \gd_{j_0} \right| [\p_{j_0}] +
      \sum_{j\neq j_0} \gd_j [\p_j]\\
      &= |\gd_{j_0}| \left( [\q_1] + [\q'_1] + \sum_{j\neq j_0} [\p_j] \right)
      + \sum_{j\neq j_0} \gd_j [\p_j]\\
      &= \left| \gd_{j_0} \right| [\q_1] + \left| \gd_{j_0} \right| [\q'_1] +
      \sum_{j\neq j_0} (\gd_j-\gd_{j_0}) [\p_j].
    \end{align*}
    As in Case~\ref{case:deltanonneg}, we let
    \begin{align*}
      M_1 &= \bigcap_{\gl_i > 0} \q_i^{\gl_i} \cap \bigcap_{\gl_i < 0}
      \left( \q'_i \right)^{|\gl_i|} \cap \bigcap_{j=1}^k
      \p_j^{\sum_{\gl_i < 0} |\gl_i|}\\
      M_2 &= \q_1^{|\gd_{j_0}|} \cap (\q'_1)^{|\gd_{j_0}|} \cap
      \bigcap_{j\neq j_0} \p_j^{\delta_j+|\delta_{j_0}|}.
    \end{align*}
    Let $r = \max\left(\{ 0 \} \cup \{\delta_j +1 \mid 1 \leq j \leq k,
    \ j \neq j_0\}\right)$.
    Under the multiplication
    map $\mu \colon M_1 \otimes_R M_2 \to M_1 M_2$, we have
    \begin{align*}
      & \hspace{15pt} \mu \left( x_{a_0b_1}^{|\gl_1|}
      x_{a_h b_{h+1}}^{\gl_{h+1}}
      \prod_{i=2}^h x_{a_{i-1}b_i}^{|\gl_i|} \otimes x_{a_0b_1}^{|\gd_{j_0}|-1}
      x_{a_hb_{h+1}}^r x_{a_{h+1}b_1} x_{a_0b_0} \right)\\
      &= x_{a_0b_1}^{|\gl_1|+|\gd_{j_0}|-1} x_{a_h b_{h+1}}^{\gl_{h+1}+r-1}
      \left( x_{a_h b_{h+1}} x_{a_{h+1}b_1} \right) x_{a_0b_0}
      \prod_{i=2}^h x_{a_{i-1}b_i}^{|\gl_i|}\\
      &= x_{a_0b_1}^{|\gl_1|+|\gd_{j_0}|-1} x_{a_h b_{h+1}}^{\gl_{h+1}+r-1}
      \left( x_{a_hb_1} x_{a_{h+1} b_{h+1}} \right) x_{a_0b_0}
      \prod_{i=2}^h x_{a_{i-1}b_i}^{|\gl_i|}\\
      &= \mu \left( x_{a_0b_1}^{|\gl_1|-1}
      x_{a_h b_{h+1}}^{\gl_{h+1}-1} x_{a_h b_1}
      \prod_{i=2}^h x_{a_{i-1}b_i}^{|\gl_i|} \otimes x_{a_0b_1}^{|\gd_{j_0}|}
      x_{a_hb_{h+1}}^r x_{a_{h+1}b_{h+1}} x_{a_0b_0} \right).
    \end{align*}
    Again $\mu$ is not injective, contradicting Fact~\ref{fact:multisiso}.
  \end{subcase}
  
  \begin{subcase}
    Suppose that $a_i+b_i \geq a_0+b_0 = a_{h+1}+b_{h+1}$ for some
    $1 \leq i \leq h$, so that $c_j+d_j > a_{h+1}+b_{h+1}$ for all
    $1 \leq j \leq k$. We can then reflect along the antidiagonal and
    reduce to Case~\ref{case:deltanonneg} or \ref{case:somedeltaneg}. \qedhere
  \end{subcase}
  
  \end{case}
\end{proof}

\begin{defn}
  Let $Y$ be a two-sided ladder. We say that $Y$ is a \textbf{spine} if:
  \begin{itemize}
    \item $h=k$, $a_1<c_1<a_2<c_2<\dots<a_h<c_h$ and $b_1>d_1>b_2>d_2>\dots>
    b_h>d_h$; or
    \item $h=k$, $c_1<a_1<c_2<a_2<\dots<c_h<a_h$ and $d_1>b_1>d_2>b_2>\dots>
    d_h>b_h$; or
    \item $h=k+1$, $a_1<c_1<a_2<c_2<\dots<a_k<c_k<a_{k+1}$ and
    $b_1>d_1>b_2>d_2>\dots>b_k>d_k>b_{k+1}$; or
    \item $k=h+1$, $c_1<a_1<c_2<a_2<\dots<c_h<a_h<c_{h+1}$ and
    $d_1>b_1>d_2>b_2>\dots>d_h>b_h>d_{h+1}$.
  \end{itemize}
\end{defn}

\begin{defn}
  Let $Y$ be a two-sided connected 2-connected ladder such that $(a_i,b_i) \nleq (c_j,d_j)$
  for all $1 \leq i \leq h$ and $1 \leq j \leq k$. We define the \textbf{spine $\ti{Y}$
  of $Y$} inductively as follows. Assume that $a_1<c_1$. If $a_1>c_1$, the
  definition is similar. Start with $u=1$ and repeat the following steps.
  \begin{enumerate}
    \item If $(a_u,b_u)$ is not an inside corner, then stop. Otherwise, suppose
    that $a_u<a_{u+1}<\dots<a_v<c_u \leq a_{v+1}$. Delete the indeterminates in
    the entries $(e,f)$, where $e<a_v$ and $f<b_u$. Update the corner
    $(a_u,b_u)$ with the values $(a_v,b_u)$. Relabel the remaining corners
    according to our conventions.
    \item If $(c_u,d_u)$ is not an inside corner, then stop. Otherwise, suppose
    that $c_u<c_{u+1}<\dots<c_v<a_{u+1} \leq c_{v+1}$. Delete the indeterminates
    in the entries $(e,f)$, where $e>c_u$ and $f>d_v$. Update the corner
    $(c_u,d_u)$ with the values $(c_u,d_v)$. Relabel the remaining corners
    according to our conventions.
    \item Update the value of $u$ to $u+1$ and repeat.
  \end{enumerate}
  We let $\ti{Y}$ be the ladder obtained when the induction stops.
\end{defn}


\begin{ex} Shown below is a two-sided ladder where $a_1>c_1$, with $h=1$, $k=4$, and its associated spine, which has only one upper inside corner (at ($c_1, d_4)$).

 \begin{center}
    \begin{picture}(152,107) 
      \put(22, 2){\line(1,0){30}}
      \put(52,2){\line(0,1){40}}
      \put(52,42){\line(1,0){10}}
      \multiput(52,44)(0,8){4}{\line(0,1){4}}
      \multiput(52,72)(8,0){4}{\line(1,0){4}}
      \put(62,42){\line(0,1){10}}
      \put(62,52){\line(1,0){10}}
      \put(72,52){\line(0,1){10}}
      \put(72,62){\line(1,0){10}}
      \put(82,62){\line(0,1){10}}
      \put(82,72){\line(1,0){30}}
      \put(112,72){\line(0,1){20}}
      \put(42,92){\line(1,0){70}}
      \put(42,22){\line(0,1){70}}
      \put(22,22){\line(1,0){20}}
      \put(22, 2){\line(0,1){20}}
    \end{picture}
  \end{center}
  
\end{ex}
\vskip.25in

\begin{thm}[Thin Ladder Theorem] \label{thm:thin} 
  Let $Y$ be a two-sided 2-connected ladder, with $h\geq 1$ lower inside corners and $k\geq1$ upper inside corners, such that $(a_i,b_i)\nleq (c_j,d_j)$
  for all $1 \leq i \leq h$ and $1 \leq j \leq k$.
  Let $R = R_2(Y)$. Then $|\s_0(R)| \leq 2$.
\end{thm}

\begin{proof}
  We prove by induction on $h+k$. 
  By Proposition~\ref{prop:4corners}, we only need to show that
  $[M_5],[M_6]$ must be trivial semidualizing modules, where
  \begin{align*}
    [M_5]&=\gl_1[\q_1] \quad \text{and}\\
    [M_6]&=\gl_{h+1}[\q_{h+1}]+\sum_{j=\kappa_1}^k \gl_{h+1}[\p_j]
    = -\gl_{h+1}[\q'_{h+1}]
    = [\omega_R]-[M_5],
  \end{align*}
  and all inside corners lie on the same antidiagonal.
  
  It suffices to show that $\gl_1\gl_{h+1} \neq 0$ leads to a
  contradiction. By reflection along the antidiagonal, we may assume that
  $a_1>c_1$. If $\gl_1>0$, then $[M_5]=[\q_1^{\gl_1}]$, and we will
  let $M_5 = \q_1^{\gl_1} = (x_{a_0b_1},x_{a_0b_1+1},\dots,
  x_{a_0b_0})^{\gl_1}$. If $\gl_1<0$, then $[M_5] = -\gl_1
  \left( [\q'_1]+\sum_{j=1}^{\kappa_2}[\p_j] \right) = \left|\gl_1\right|
  [(x_{a_0b_1},x_{a_0+1b_1},\dots,x_{c_1b_1})]$. We then
  let $M_5=(x_{a_0b_0},x_{a_0+1b_0},\dots,x_{c_1b_0})^{|\gl_1|}$.
  Similarly, if $\gl_{h+1}>0$, then we let $M_6=(x_{a_{h+1}b_{h+1}},
  x_{a_{h+1}b_{h+1}+1},\dots,x_{a_{h+1}d_k})^{\gl_{h+1}}$,
  and if $\gl_{h+1}<0$, then finally we let $M_6=
  (x_{a_hb_{h+1}},x_{a_h+1b_{h+1}},\dots,x_{a_{h+1}b_{h+1}})^{|\gl_{h+1}|}$.
  
  Consider the case when $\gl_1,\gl_{h+1}<0$. Let $\ti{Y}$ be the spine of
  $Y$. We construct part of a minimal free resolution of $M_5$ over
  $\ti{R}=R_2(\ti{Y})$ given by
  \[
    0 \xleftarrow{\partial_0} \ti{R}^{\gb_0}
    \xleftarrow{\partial_1} \ti{R}^{\gb_1}
    \xleftarrow{\ti{\partial}_2} \ti{R}^{\ti{\gb}_2}
    \xleftarrow{\ti{\partial}_3} \cdots
    \xleftarrow{\ti{\partial}_{2\ti{h}-1}} \ti{R}^{\ti{\gb}_{2\ti{h}-1}}
    \xleftarrow{\ti{\partial}_{2\ti{h}}} \ti{R}^{\ti{\gb}_{2\ti{h}}}
    \xleftarrow{\ti{\partial}_{2\ti{h}+1}} \cdots,
  \]
  where
  {\allowdisplaybreaks
  \begin{align*}
    \partial_0 &= \left( \begin{matrix}
      x_{a_0b_0}^{|\gl_1|} & x_{a_0b_0}^{|\gl_1|-1}x_{a_0+1b_0} & \cdots
    \end{matrix} \right),\\
    \partial_1 &= \left( \begin{matrix}
      x_{a_0+1\ti{d}_1-1} & -x_{a_0+1\ti{d}_1} & \cdots\\
      -x_{a_0\ti{d}_1-1} & x_{a_0\ti{d}_1} & \cdots\\
      0 & 0 & \cdots\\
      \vdots & \vdots & \cdots\\
      0 & 0 & \cdots
    \end{matrix} \right),\\
    \ti{\partial}_2 &= \left( \begin{matrix}
      x_{\ti{a}_1\ti{d}_1} & x_{\ti{a}_1+1\ti{d}_1} & \cdots\\
      x_{\ti{a}_1\ti{d}_1-1} & x_{\ti{a}_1+1\ti{d}_1-1} & \cdots\\
      0 & 0 & \cdots\\
      \vdots & \vdots & \cdots\\
      0 & 0 & \cdots
    \end{matrix} \right),\dots\\
    \ti{\partial}_{2\ti{h}-1} &= \left( \begin{matrix}
      x_{\ti{a}_{\ti{h}-1}+1\ti{d}_{\ti{h}}-1} &
      -x_{\ti{a}_{\ti{h}-1}+1\ti{d}_{\ti{h}}} & \cdots\\
      -x_{\ti{a}_{\ti{h}-1}\ti{d}_{\ti{h}}-1} &
      x_{\ti{a}_{\ti{h}-1}\ti{d}_{\ti{h}}} & \cdots\\
      0 & 0 & \cdots\\
      \vdots & \vdots & \cdots\\
      0 & 0 & \cdots
    \end{matrix} \right),\\
    \ti{\partial}_{2\ti{h}} &= \left( \begin{matrix}
      x_{\ti{a}_{\ti{h}}\ti{d}_{\ti{h}}} &
      x_{\ti{a}_{\ti{h}}+1\ti{d}_{\ti{h}}} & \cdots\\
      x_{\ti{a}_{\ti{h}}\ti{d}_{\ti{h}}-1} &
      x_{\ti{a}_{\ti{h}}+1\ti{d}_{\ti{h}}-1} & \cdots\\
      0 & 0 & \cdots\\
      \vdots & \vdots & \cdots\\
      0 & 0 & \cdots
    \end{matrix} \right)
    = \left( \begin{matrix}
      x_{a_h\ti{d}_{\ti{h}}} &
      x_{a_h+1\ti{d}_{\ti{h}}} & \cdots\\
      x_{a_h\ti{d}_{\ti{h}}-1} &
      x_{a_h+1\ti{d}_{\ti{h}}-1} & \cdots\\
      0 & 0 & \cdots\\
      \vdots & \vdots & \cdots\\
      0 & 0 & \cdots
    \end{matrix} \right),\dots
  \end{align*}}
  This minimal free resolution forms part of a minimal free resolution of $M_5$
  over $R=R_2(Y)$ given by
  \[
    F_{\bullet} = 0 \xleftarrow{\partial_0} R^{\gb_0}
    \xleftarrow{\partial_1} R^{\gb_1}
    \xleftarrow{\partial_2} R^{\gb_2}
    \xleftarrow{\partial_3} \cdots
    \xleftarrow{\partial_{2\ti{h}-1}} R^{\gb_{2\ti{h}-1}}
    \xleftarrow{\partial_{2\ti{h}}} R^{\gb_{2\ti{h}}}
    \xleftarrow{\partial_{2\ti{h}+1}} \cdots,
  \]
  with $\ti{\partial}_2,\ti{\partial}_3,\dots$ being represented by upper
  left submatrices of the matrices representing $\partial_2,\partial_3,\dots$.
  In $F_{\bullet} \otimes M_6$, we have
  \[
    \mathbf{x} = \left( \begin{matrix}
      x_{a_h+1b_{h+1}}^{|\gl_{h+1}|}
      & -x_{a_h+1b_{h+1}}^{|\gl_{h+1}|-1} x_{a_hb_{h+1}}
      & 0 & \cdots & 0 \end{matrix} \right)^T \in \ker(\partial_{2\ti{h}}).
  \]
  However, $\mathbf{x} \notin \im(\partial_{2\ti{h}+1})$ since $F_{\bullet}$ is
  a minimal resolution. Hence $\tor_{2\ti{h}}^R(M_5,M_6) \neq 0$, contradicting
  Fact~\ref{fact:tensorgivesomega}.
  
  Now suppose that $\kappa_1\neq k$. If $\gl_1,\gl_{h+1}>0$, then we also use
  $\tor_{2\ti{h}}^R(M_5,M_6)$ to reach a contradiction. If $\gl_1>0$ and
  $\gl_{h+1}<0$, then we use $\tor_{2\ti{h}-1}^R(M_5,M_6)$; and if $\gl_1<0$
  and $\gl_{h+1}>0$, then we use $\tor_{2\ti{h}+1}^R(M_5,M_6)$.
  
  Finally, suppose that $\kappa_1=k$. If $\gl_1\gl_{h+1}<0$, then we use
  $\tor_{2\ti{h}-1}^R (M_5,M_6)$; and if $\gl_1,\gl_{h+1}>0$, then we use
  $\tor_{2\ti{h}-2}^R(M_5,M_6)$ to reach a contradiction.
\end{proof}

\begin{thm}[Two-Sided Ladder Theorem] \label{thm:mixed} 
  Let $Y$ be a 2-connected ladder, with $h$ lower inside corners and $k$ upper inside corners, such that $(a_i,b_i)\neq (c_j,d_j)$
  for all $1 \leq i \leq h$ and $1 \leq j \leq k$.
  Then $|\s_0(R_2(Y))| \leq 2$.
\end{thm}

\begin{proof}
  We will argue by induction on $h+k$. By Theorem~\ref{thm:onesided},
  we may assume that $h,k>0$. The case $h=k=1$ is given by
  Theorems~\ref{thm:thick} and \ref{thm:thin}. In the induction step,
  by Proposition~\ref{prop:4corners} we only need to show that $[M_1],[M_2]$ must be
  trivial semidualizing modules, where
  \[
    [M_1] = \sum_{i=\eta_1}^{\eta_2+1} \gl_i[\q_i] \quad \text{and} \quad
    [M_2] = \sum_{j=\kappa_1}^{\kappa_2} \gd_j[\p_j] = [\omega_R]-[M_1],
  \]
  $\eta_1\leq \eta_2$, $\kappa_1\leq \kappa_2$, $\eta_1,\kappa_1$ are not both 1 (equivalently, we
  cannot have both $\eta_2=h$ and $\kappa_2=k$), and the corners $(a_0,b_0)$,
  $(a_{h+1},b_{h+1})$ together with all inside corners, except $(a_i,b_i),
  (c_j,d_j)$ for $\eta_1 \leq i \leq \eta_2$ and $\kappa_1 \leq j \leq \kappa_2$,
  all lie on the same antidiagonal.
  
  Note that we cannot have both $\eta_2\neq h$ and $\kappa_2\neq k$. Otherwise, we would
  have $a_h\leq c_{\kappa_1} \leq c_{\kappa_2} < c_k$ and $b_h < b_{\eta_2} \leq
  b_{\eta_1} \leq d_k$, a contradiction since $(a_h,b_h)$ and $(c_k,d_k)$ should
  lie on the same antidiagonal.
  
  By reflection along the antidiagonal, we may assume that $\eta_2 \neq h$ and
  $\kappa_2=k$. In this case, we have $a_i \leq a_{\eta_2} \leq c_1 \leq c_k$
  for all $1 \leq i \leq \eta_2$, and $b_i \leq b_1 \leq d_k \leq d_1$ for all
  $i \geq 1$. So $(a_i,b_i) \lneqq (c_j,d_j)$ for all $1 \leq i \leq \eta_2$ and
  $1\leq j \leq k$.
  
  We also note that for $\kappa_1 \leq j \leq \kappa_2=k$, we have $c_j \geq a_h$ and
  $d_j \geq d_k \geq b_1 > b_h$. Hence $c_j+d_j>a_h+b_h=a_{h+1}+b_{h+1}$ for all
  $\kappa_1 \leq j \leq k$.
  
  Now suppose that $a_i+b_i\geq a_0+b_0=a_{h+1}+b_{h+1}$, where $1 \leq i
  \leq \eta_2$. Since $\eta_2 \neq h$, i.e.\ $\kappa_1 \neq 1$, the corners $(c_1,d_1)$,
  $(a_0,b_0)$ and $(a_{h+1},b_{h+1})$ lie on the same antidiagonal. This is
  a contradiction, since $(a_i,b_i) \lneqq (c_1,d_1)$. Hence $a_i+b_i
  <a_0+b_0=a_{h+1}+b_{h+1}$ for all $1 \leq i \leq \eta_2$.
  
  Finally, let $\ti{Y}$ be the ladder obtained by deleting from $Y$ columns
  $b_{h+1},b_{h+1}+1,\dots,b_{\eta_2+1}-1$. We can then use arguments similar to
  those in Theorem~\ref{thm:thick}, Case~\ref{case:somedeltaneg} on $\ti{Y}$
  to finish the induction.
\end{proof}

We will generalize Theorem~\ref{thm:onesidedpathconnected} and
Theorem~\ref{thm:mixed} in \cite{SWSeSp2}. Here we end with some examples
to illustrate our results and point to our future work.

\begin{ex} \label{multipleladders} We consider $R = R_t(-)$ for the ladders shown earlier.
\begin{enumerate}
\item If \ref{eq180304a} is the ladder on page~\pageref{eq180304a}, then $\s_0(R_2(L)) = \{[R], [\omega_R] \}$, by Theorem \ref{thm:mixed} and \cite[Proposition 2.5]{Co};
\item If $O$ is the one-sided ladder on page~\pageref{ladders:OT}, then $\s_0(R_3(O)) = \{[R]\}$ and $\s_0(R_2(O)) = \{[R], [\omega_R]\}$, by Theorem \ref{thm:onesided} and \cite[Example 4.10]{Co};
\item For the ladder $T$ with a coincidental inside corner on page~\pageref{ladders:OT}, we show in \cite{SWSeSp2} that $\s_0(R_2(T)) = \{ [R], [\omega_R], [(x_{12}, x_{13})], [(x_{31}, x_{32})] \}$.
\end{enumerate}
\end{ex}

\providecommand{\bysame}{\leavevmode\hbox to3em{\hrulefill}\thinspace}
\providecommand{\MR}{\relax\ifhmode\unskip\space\fi MR }
\providecommand{\MRhref}[2]{%
  \href{http://www.ams.org/mathscinet-getitem?mr=#1}{#2}
}
\providecommand{\href}[2]{#2}

\end{document}